\documentclass[reqno]{amsart}
\usepackage{amssymb}
\usepackage{amsmath}
\usepackage{amsthm}
\usepackage{mathtools}
\usepackage[svgnames]{xcolor}
\usepackage[pdftex]{hyperref}

\hypersetup{
  pdftitle={},
  pdfauthor={Stefan Le Coz <slecoz@math.univ-toulouse.fr>, Yifei Wu <yifei@bnu.edu.cn>},
  pdfsubject={},
  pdfkeywords={}, 
  colorlinks=true,
  linkcolor=DarkBlue  ,          
  citecolor=DarkRed,        
  filecolor=DarkMagenta,      
  urlcolor=DarkGreen,           
}

\newtheorem{theorem}{Theorem}[section]
\newtheorem{proposition}[theorem]{Proposition}
\newtheorem{lemma}[theorem]{Lemma}
\newtheorem*{notation}{Notation}

\theoremstyle{remark}
\newtheorem{remark}[theorem]{Remark}
\newtheorem{claim}[theorem]{Claim}

\DeclarePairedDelimiter{\norm}{\lVert}{\rVert}
\DeclarePairedDelimiter{\abs}{\lvert}{\rvert}

\newcommand{\psld}[2]{\left( #1,#2 \right)_{2}}

\newcommand{\dual}[2]{\left\langle #1,#2 \right\rangle}

\newcommand{\eps}{\varepsilon}

\newcommand{\R}{\mathbb{R}}

\DeclareMathOperator{\spn}{span}
\DeclareMathOperator{\Mod}{Mod}

\renewcommand{\leq}{\leqslant}
\renewcommand{\geq}{\geqslant}

\DeclareMathAlphabet{\mathpzc}{OT1}{pzc}{m}{it}
\renewcommand{\Re}{\mathcal R\!\mathpzc{e}}
\renewcommand{\Im}{\mathcal I\!\mathpzc{m}}

\usepackage{color}

\begin{document}

\title[Stability of multi-solitons for dNLS]{Stability of
  multi-solitons for the\\ derivative nonlinear Schr\"odinger equation}

\author[S.~Le Coz]{Stefan Le Coz}
\thanks{The work of S. L. C. is 
  partially supported by ANR-11-LABX-0040-CIMI within the
  program ANR-11-IDEX-0002-02 and  ANR-14-CE25-0009-01}

\author[Y.~Wu]{Yifei Wu}

\address[Stefan Le Coz]{Institut de Math\'ematiques de Toulouse,
  \newline\indent
  Universit\'e Paul Sabatier
  \newline\indent
  118 route de Narbonne, 31062 Toulouse Cedex 9
  \newline\indent
  France}
\email[Stefan Le Coz]{slecoz@math.univ-toulouse.fr}

\address[Yifei Wu]{Center for Applied Mathematics,
  \newline\indent
  Tianjin University
  \newline\indent
  Tianjin 300072
  \newline\indent
  China}
\email[Yifei Wu]{yerfmath@gmail.com}


\date{\today}

\begin{abstract}
  The nonlinear Schr\"odinger equation with derivative cubic
  nonlinearity admits a family of solitons, which are 
  orbitally stable in the energy space. In this work, we prove the
  orbital stability of multi-solitons configurations in the energy space, under
  suitable assumptions on the speeds and frequencies of the composing
  solitons. The main ingredients of the proof are modulation theory,
  energy coercivity and
  monotonicity properties. 
\end{abstract}

\maketitle

\tableofcontents

\section{Introduction}

We consider the nonlinear Schr\"odinger equation with
derivative nonlinearity:
\begin{equation}\tag{dNLS}\label{eq:nls}
  iu_t+u_{xx}+i|u|^2u_x=0.
\end{equation}
The unknown function $u$ is a complex-valued function of time~$t\in\R$
and space~$x\in\R$.

The derivative nonlinear Schr\"odinger equation was originally introduced in Plasma Physics as a simplified
model for Alfv\'en waves propagation, see~\cite{Mj76,SuSu99}. Since
then, it has attracted a lot of attention from the mathematical
community. Let us give a few examples. It was first studied using
integrable methods by Kaup and Nevel~\cite{KaNe78}. Later on, Hayashi
and Ozawa~\cite{Ha93,HaOz92} obtained local well-posedness in the energy space
$H^1(\R)$. 
Local well-posedness in low regularity spaces $H^s(\R)$, $s\geq 1/2$
was investigated by Takaoka~\cite{Ta01}. The problem of global
well-posedness for small mass initial data in low regularity spaces has
attracted the attention of a number of authors, see e.g.~\cite{CoKeStTa01,CoKeStTaTa02,GuWu16,MiWuXu11}. 
When considered on the
torus, local existence in $H^s(\mathbb T)$ was proved by Herr~\cite{He06}
for $s\geq 1/2$. 
Results for $s<1/2$ were recently obtained by Takaoka~\cite{Ta16}. A
probabilistic approach to local existence was initiated by Thomann and
Tzvetkov~\cite{ThTz10}.
Recently, Wu~\cite{Wu13,Wu15}
proved global existence in $H^1(\R)$ for initial data $u_0$ having
mass~$M(u_0)=\frac12\norm{u_0}_{L^2}^2$ less than threshold $2\pi$. The method introduced by Wu was extend to the
torus by Mosincat and Oh~\cite{MoOh15}. Despite the amount of studies devoted to
\eqref{eq:nls}, existence of blowing up solutions remains a totally open
problem. Global existence was recently investigated using
integrability techniques by Liu, Perry and
Sulem~\cite{LiPeSu15,LiPeSu16}
and by Pelinovsky and Shimabukuro~\cite{PeSh16}. Analysis of singular
profiles in a supercritical version of~\eqref{eq:nls} was performed by
Cher, Simpson and Sulem~\cite{ChSiSu16}. 

Before presenting our results, we start with some preliminaries. 

Under gauge transformations, \eqref{eq:nls} may take various (equivalent) forms. In
particular, if 
\[
v(t,x)=\exp\left(-\frac i2\int_{-\infty}^x|u(t,y)|^2dy\right)u(t,x),
\] 
then $v$ solves 
\begin{equation}
\label{eq:kaup-newell}
iv_t+v_{xx}+i(|v|^2v)_x=0.
\end{equation}
Alternatively, setting 
\[
w(t,x)=\exp\left(\frac i4\int_{-\infty}^x|u(t,y)|^2dy\right)u(t,x),
\] 
then $w$ solves 
\begin{equation}
\label{eq:Gerdzhikov-Ivanov}
iw_t+w_{xx}-iw^2\bar w_x+\frac{1}{2}|w|^4w=0.
\end{equation}
Under the form \eqref{eq:nls}, the derivative nonlinear Schr\"odinger
equation is sometimes referred to as the Chen-Liu-Lee
equation~\cite{ChLeLi79}. The form \eqref{eq:kaup-newell} might be called the Kaup-Newell
equation~\cite{KaNe78}. The form \eqref{eq:Gerdzhikov-Ivanov} is the Gerdzhikov-Ivanov
equation~\cite{GeIv83}. Yet another notable (but apparently not christened) form of \eqref{eq:nls} is obtained setting 
\[
z(t,x)=\exp\left(\frac i2\int_{-\infty}^x|u(t,y)|^2dy\right)u(t,x).
\] 
Then $z$ solves 
\begin{equation}
\label{eq:Wu}
iz_t+z_{xx}-\frac i2|z|^2z_x+\frac i2z^2\bar z_x+\frac{3}{16}|z|^4z=0.
\end{equation}
The equation \eqref{eq:Wu} played a central role in the papers on
global well-posedness \cite{Wu13,Wu15}.
Since all these derivative nonlinear Schr\"odinger equations are related via gauge
transformations, any result on one of the forms can a priori be
transfered to the other forms. Depending on the aim, some form usually
turns
out to be much easier to work with than the others. In this paper, we will
mostly use the form \eqref{eq:nls}.

Interestingly, given a solution $u$ of \eqref{eq:nls} and
$\lambda>0$, then
\[
u_\lambda(t,x)=\frac{1}{\sqrt{\lambda}}u\left(\frac{t}{\lambda^2},\frac{x}{\lambda}\right)
\]
is also a solution of \eqref{eq:nls}. In particular, \eqref{eq:nls} is
$L^2$-critical. However, its behavior widely differs from the one of its
celebrated power-type counter part, the $1d$ quintic nonlinear
Schr\"odinger equation. In particular, \eqref{eq:nls} is not invariant
by the pseudo-conformal transformation and no explicit (and in fact
not at all) blow-up solution is known for \eqref{eq:nls}. Another main
difference between \eqref{eq:nls} and quintic NLS is that the former
one is \emph{not} invariant by a Galilean transform. 

The equation~\eqref{eq:nls} can be written in Hamiltonian form as 
\[
iu_t=E'(u),
\]
where the \emph{Hamiltonian} (or \emph{energy}) is given by 
\[
E(u)=\frac12\norm{u_x}_{L^2}^2+\frac14\Im\int_{\R}|u|^2\bar u u_xdx.
\]
At least formally, the Hamiltonian $E$ is conserved along the flow of
\eqref{eq:nls}. In addition, two other quantities are conserved: the
mass and the momentum, defined by 
\[
M(u)=\frac12\norm{u}_{L^2}^2,\quad P(u)=\frac12\Im\int_{\R}u\bar u_xdx.
\]
Note that~\eqref{eq:nls} is an integrable equation (see e.g.~\cite{KaNe78}) and there exists in fact an
infinity of conservation laws (see e.g.~\cite{TsFu81}). However, our
goal here is to study the properties of the solutions of
\eqref{eq:nls} with a robust method not relying on its algebraic
peculiarities. Our approach may be applied to other similar
non-integrable equations,
e.g. the generalized derivative nonlinear Schr\"odinger equations
considered in~\cite{AbLiSu11,Fu16,LiSiSu13-2,LiSiSu13,Oh14}.

As is now well-known, given real parameters $\omega>0$ and
$-2\sqrt{\omega} <c\leq 2\sqrt{\omega}$, there exist
traveling waves solutions of~\eqref{eq:nls} of the form
\[
R_{\omega,c}(t,x)=e^{i\omega t}\phi_{\omega,c}(x-ct). 
\]
The profiles $\phi_{\omega,c}$ are unique  up to phase shifts and
translations (see e.g.~\cite{CoOh06}) and are given by an explicit
formula (see Section~\ref{sec:solitons} for details). The stability of the solitary waves of~\eqref{eq:nls} was considered in
\cite{CoOh06,GuWu95,KwWu16}. In particular, in~\cite{CoOh06}, Colin and Ohta proved that for any $(\omega,c)\in\R^2$ with
$c^2<4\omega$, the solitary wave $R_{\omega,c}$
is orbitally stable. The orbital stability of the lump soliton in the case
$c=2\sqrt{\omega}$ was considered very recently by Kwon and Wu in~\cite{KwWu16}. The stability theory for a more general version
of~\eqref{eq:nls} was developed in~\cite{LiSiSu13} by Liu, Simpson and
Sulem.

\emph{Multi-solitons} are solutions to~\eqref{eq:nls} that behave at
large time like a sum
of solitons. They can be proved to exist by inverse scattering
transform (see~\cite{KaNe78} for~\eqref{eq:nls} or~\cite{ZaSh72} for the cubic nonlinear Schr\"odinger
equation), using energy methods (see~\cite{CoLe11,CoMaMe11,MaMe06} for
the nonlinear Schr\"odinger equation or e.g.~\cite{DeLeWe16,IaLe14}
for Schr\"odinger systems) or
with fixed point arguments (see~\cite{CoLe11,LeLiTs15,LeTs14} for the nonlinear Schr\"odinger equation). 

As each individual
solitary wave of~\eqref{eq:nls} is stable, it is reasonable to investigate the
stability of a sum of solitary waves; this will be our goal in this
paper. 
Precisely, we want to show that, under some
conditions, if the initial data is close to a sum of solitons
profiles, then the associated solution of~\eqref{eq:nls} will behave for any
positive time as a sum of solitons. To our
knowledge, no information was available so far on the stability of these
multi-solitons configurations.

To study the stability of multi-solitons, several
approaches are possible. One can work in weighted spaces and get
asymptotic stability results for multi-solitons configurations
\cite{Pe04,RoScSo03}. An alternative approach, when the underlying
equation is integrable, is to take advantage of the integrable
structure to obtain stability (in a relatively restricted class of
functions), see e.g.~\cite{GeZh09}. Finally, one can work in the
energy space, and this is what will be done in this paper. We will
prove a result in the same family as the results obtained by Martel,
Merle and Tsai for the Kortweg-de Vries equation~\cite{MaMeTs02} and the twisted
nonlinear Schr\"odinger equation~\cite{MaMeTs06}. 
This approach was later extended to the Gross-Pitaevskii equation by Bethuel, Gravejat and
Smets~\cite{BeGrSm14} and to the Landau-Lifshitz equation by de Laire
and Gravejat~\cite{deGr15}. Instability results are available in~\cite{Co14,CoLe11}.

Our main result is the following.

\begin{theorem}\label{thm:1}
  Let $N\in\mathbb N$. For $j=1,\dots,N$  let
  $\omega_j\in(0,\infty)$, $c_j\in(-2\sqrt{\omega_j},2\sqrt{\omega_j})$, $x_j\in\R$
  and $\theta_j\in\R$. Let $(\phi_j)=(\phi_{\omega_j,c_j})$ be the
  corresponding solitary wave profiles given by the explicit
  formula~\eqref{eq:16}. For $j=2,\dots,N$, let 
  \[
  \sigma_j=2\frac{\omega_j-\omega_{j-1}}{c_j-c_{j-1}}.
  \]
  Assume that for  $j=2,\dots,N$ we have
  \begin{equation}
    \label{eq:speed-frequency-ratios}
    \sigma_j>0\quad\text{and}\quad c_{j-1}<\sigma_j<c_j.
  \end{equation}
  Then there exist $\alpha>0$, $L_0>0$, $A_0>0$ and $\delta_0>0$ such that for any
  $u_0\in H^1(\R)$, $L>L_0$ and $0<\delta<\delta_0$, the following
  property is satisfied. If 
  \[
  \norm[\bigg]{u_0-\sum_{j=1}^Ne^{i\theta_j}\phi_j(\cdot-x_j)}_{H^1}\leq \delta
  \]
  and if for all $j=1,\dots,N-1$,
  \[
  x_{j+1}-x_j>L,
  \]
  the solution $u$ of~\eqref{eq:nls} with $u(0)=u_0$ is globally defined in
  $H^1(\R)$ for $t\geq 0$, and there exist functions
  $\tilde x_1(t),\dots,\tilde x_N(t)\in\R$, and 
  $\tilde \theta_1(t),\dots,\tilde \theta_N(t)\in\R$, such that for all $t\geq 0$, 
  \[
  \norm[\bigg]{u(t)-\sum_{j=1}^Ne^{i\tilde\theta_j(t)}\phi_j(\cdot-\tilde
    x_j(t))}_{H^1}\leq A_0\left(\delta+e^{-\alpha L}\right).
  \]
\end{theorem}

\begin{remark}
  We can further describe the behavior of the functions
  $\tilde\theta_j$ and $\tilde x_j$, see Proposition~\ref{prop:modulation}. In
  particular, they are of class $\mathcal C^1$
  and verify the dynamical laws
  \[
  \partial_t\tilde x_j\sim c_j,\quad \partial_t\tilde \theta_j\sim \omega_j.
  \]
  Hence the behavior of $e^{i\tilde\theta_j(t)}\phi_j(\cdot-\tilde
  x_j(t))$ is close to the one of $R_{\omega_j,c_j}(t)$.
\end{remark}

\begin{remark}
  As we are trying to prove a stability result for the multi-solitons
  configuration, it natural to assume that we start with well ordered
  solitary waves, i.e. 
  \[
  c_1<c_2<\cdots<c_N,\quad\text{ and }\quad x_1<x_2<\cdots<x_N.
  \]
  This prevents the crossing of solitary waves at a later time. 

  As in~\cite{MaMeTs06}, the stronger condition~\eqref{eq:speed-frequency-ratios} that we impose on the parameters of
  the waves is needed for technical purposes. 
  
  A consequence of~\eqref{eq:speed-frequency-ratios}  is that 
  \[
  \omega_1<\omega_2<\cdots<\omega_N.
  \]
  In other words, the larger
  the amplitude is, the faster the soliton should travel. This is
  somehow reminiscent from what happens in the context of the
  Korteweg-de Vries (KdV) equation, where speed and amplitude are controlled
  by the same parameter~\cite{MaMeTs02}. Another similitude with the
  KdV equation, consequence of~\eqref{eq:speed-frequency-ratios}, is
  that our solitons should all be traveling to the right (i.e. with positive
  speeds), except for the first one, which is allowed to travel to the
  left (i.e. with negative speed). 

  The speed/frequency ratio  condition is different from the
  equivalent one in~\cite[condition (A3)]{MaMeTs06}. In particular,
~\cite[condition (A3)]{MaMeTs06} allows for the solitary waves to
  have equal frequencies, in which case the only condition for the
  speeds is to be strictly increasing. This is ruled out by
~\eqref{eq:speed-frequency-ratios}.
  
  An example of a range of parameters verifying
~\eqref{eq:speed-frequency-ratios} is given by
  \[
  \omega_j=j^2+1,\quad c_j=2j, \quad j=1,\dots,N.
  \]
  It is not hard to construct many other examples.
\end{remark}

\begin{remark}
  The constant $\alpha$ in Theorem~\ref{thm:1} can be made explicit:
  $\alpha=\frac{1}{32}\omega_\star $, where $\omega_\star$
  is the minimal decay rate of the solitons defined in
~\eqref{eq:def-omega}.
\end{remark}

\begin{remark}
  Our theorem does not cover the case where one of the solitons is a
  lump soliton, i.e. when $c_j=2\sqrt{\omega_j}$ for some
  $j=1,\dots,N$. Indeed, lump solitons are significantly different
  from the other solitons (algebraic decay instead of exponential
  decay, weaker stability, etc.), which prevent to include them in our
  analysis. 
\end{remark}

\begin{remark}
  The main technical differences between
  (NLS) and~\eqref{eq:nls} are that the later one is not any more
  Galilean invariant and there is no scaling between
  solitons. Nevertheless
  the proof of our result is largely inspired by the proof of
~\cite[Theorem~1]{MaMeTs06}. 
  As far as possible, we have tried to
  keep the same (or similar) notations.  
\end{remark}

Our strategy for the proof of Theorem~\ref{thm:1} is, as in~\cite{MaMeTs02,MaMeTs06}, the following. We
use a bootstrap argument, which goes as follows.  Assume that an
initial data $u_0$
is located close enough to a sum of soliton profiles, and that
the associated solution $u$ to~\eqref{eq:nls} stays until some time $T$ in a neighborhood  of size $\eps$ of a sum of (modulated)
solitons profiles. The bootstrap argument tells us that,
in fact, $u$ stays until the same time $T$ in \emph{a neighborhood  of size $\eps/2$} of a sum of (modulated)
solitons profiles. This allows us to extend the time $T$ up to
$\infty$ and proves the stability of the configuration. To obtain the
bootstrap result, we rely on several ingredients. First, we need a
modulation result around the soliton profiles. As usual, modulation is
obtained via the Implicit Function Theorem, with here the particularity
that we rely on explicit calculations to prove invertibility of the
Jacobian. This is a specific feature of~\eqref{eq:nls} which allows us
to modulate in a more natural way than in~\cite{MaMeTs06}. The second
ingredient is a coercivity property for a linearized action
functional. The functional is based on the linearized action around
each soliton, which we prove to be coercive (up to some orthogonality
conditions). The third ingredient is a series
of monotonicity properties for suitably localized mass/momentum functionals. These monotonicity
properties are involved in a crucial way in the control of the
modulation parameters. 

The rest of the paper is organized as follows. In Section
\ref{sec:solitons} we review and develop the stability theory for a single
solitary wave. In particular, we obtain a coercivity property for the linearized
action functional around a single soliton using variational
characterizations. In Section~\ref{sec:bootstrap-argument} we state
the bootstrap argument and prove Theorem~\ref{thm:1}. Section
\ref{sec:modulation} is devoted to the modulation result. In section
\ref{sec:local-proc} we derive the monotonicity properties of
localized mass/momentum functionals. Section~\ref{sec:line-acti-funct} deals with the
construction of the action-like linearized functional for the sum of
solitons and its coercivity properties. In Section
\ref{sec:contr-modul-param} we control the modulation parameters using
the monotonicity properties. Finally, in Section
\ref{sec:proof-bootstr-result} we prove the bootstrap result. The
Appendix~\ref{sec:some-techn-results} contains explicit formulas that we use at several occasions
in the paper.

\begin{notation}
  The space $L^2(\R)$ is considered as a real Hilbert space with the
  scalar product
  \[
  \psld{u}{v}=\Re\int_{\R}u\bar vdx.
  \]
  Whenever an inequality is true up to a positive constant, we use the
  notation $\gtrsim$ or $\lesssim$. 
  Throughout the paper, the letter $C$ will denote various positive
  constants whose exact value may vary from line to line but is of no
  importance in the analysis. 
\end{notation}

\section{Solitary Waves and Stability Theory}\label{sec:solitons}

We will need for
the study of the stability of a sum of solitary waves some tools coming from the
stability theory of a single solitary wave. These tools are however
not immediately available in the literature  and we need to introduce
them ourselves in this section. We believe that the results presented
in this section are of independent interest and may be useful for
further studies of~\eqref{eq:nls}.

Recall that given parameters $\omega>0$ and $|c|<2\sqrt{\omega}$, there exist
traveling waves solutions of~\eqref{eq:nls} of the form
\[
R_{\omega,c}(t,x)=e^{i\omega t}\phi_{\omega,c}(x-ct). 
\]
The profile $\phi_{\omega,c}$ is unique  up to phase shifts and
translations (see e.g.~\cite{CoOh06}) and is given by the explicit formula
\begin{equation}
  \label{eq:16}
  \phi_{\omega,c}(x)=\varphi_{\omega,c}(x)\exp\left(\frac c2 ix-\frac i4\int_{-\infty}^x|\varphi_{\omega,c}(\xi)|^2d\xi\right)
\end{equation}
where
\begin{equation}
  \label{eq:varphi-formula}
  \varphi_{\omega,c}(x)=\left(\frac{\sqrt{\omega}}{4\omega-c^2}\left(\cosh\left(x\sqrt{4\omega-c^2}\right)-\frac{c}{2\sqrt{\omega}}\right)\right)^{-\frac12}.
\end{equation}
The profile is also the unique  (up to phase shifts and translations)
solution to the elliptic ordinary differential equation 
\begin{equation}\label{eq:profile}
  -\phi_{xx}-i|\phi|^2\phi_x+\omega\phi+ic\phi_x=0,
\end{equation}
and it is also a critical point of the action functional $S$, 
\[
S=S_{\omega,c}=E+\omega M+c P.
\]
For future reference, note that the function $\varphi_{\omega,c}$ verifies the equation
\[
-\varphi_{xx}+\left(\omega-\frac{c^2}{4}\right)\varphi-\frac12\Im\left(\varphi\bar\varphi_x\right)\varphi+\frac{c}{2}|\varphi|^2\varphi-\frac{3}{16}|\varphi|^4\varphi=0,
\]
or, since $\varphi_{\omega,c}$ is real, 
\begin{equation}\label{eq:varphi}
  -\varphi_{xx}+\left(\omega-\frac{c^2}{4}\right)\varphi+\frac{c}{2}|\varphi|^2\varphi-\frac{3}{16}|\varphi|^4\varphi=0.
\end{equation}
From the explicit formula~\eqref{eq:16}, we see that $\phi_{\omega,c}$
is exponentially decaying. Precisely, for any $\alpha<1$ we have
\begin{equation}\label{eq:exp-decay}
  |\partial_x\phi_{\omega,c}(x)|+|\phi_{\omega,c}(x)|\leq C_\alpha e^{-\alpha\sqrt{\omega-\frac{c^2}{4}}|x|}.
\end{equation}
Note that this decay is not affected by a Gauge transform or a
Galilean transform.

The main result of this section is the following coercivity property. 

\begin{proposition}[Coercivity for one solitary wave]
  \label{prop:3}
  For any $\omega, c\in\R$ with $4\omega>c^2$, there exists $\mu\in\R$
  such that
  for
  any $\eps\in H^1(\R)$ verifying the orthogonality conditions 
  \[
  \psld{\eps}{i\phi_{\omega,c}}=
  \psld{\eps}{\partial_x\phi_{\omega,c}}=
  \psld{\eps}{\phi_{\omega,c}+i\mu\partial_x\phi_{\omega,c}}=
  0,
  \]
  we have 
  \[
  H_{\omega,c}(\eps):=\dual{S_{\omega,c}''(\phi_{\omega,c})\eps}{\eps}\gtrsim \norm{\eps}_{H^1}^2.
  \]
\end{proposition}

\begin{remark}
  If $c<0$, then we can choose $\mu=0$. 
\end{remark}

\begin{remark}
  Following the classical approach by Weinstein~\cite{We85} (see
  e.g.~\cite{Le09} for an introduction),
  one obtains as a corollary of Proposition~\ref{prop:3} the orbital
  stability of solitary waves. 
\end{remark}

\begin{notation}
  Many quantities defined in this paper will depend on $\omega$ and
  $c$. For the sake of clarity in notation, we shall very often drop the
  subscript $\omega,c$ and dependency in $\omega$ and $c$ will be only
  understood. 
\end{notation}

For future reference, we give the explicit expression of $S''(\phi)\eps$:
\begin{equation}\label{eq:explicit}
  S''(\phi)\eps=-\partial_{xx}\eps-i|\phi|^2\partial_x\eps-2i\Re(\phi \bar\eps)\phi_x+\omega\eps+ic\partial_x\eps,
\end{equation}
and the explicit expression of $H$:
\[
H(\eps)=\norm{\eps_x}_{L^2}^2+\Im\int_\R|\phi|^2\bar\eps
\eps_xdx+2\int_\R\Re(\phi \bar \eps)\Im(\phi_x \bar\eps)dx+\omega\norm{\eps}_{L^2}^2 +c\Im\int_{\R}\eps\bar\eps_xdx.
\]

The proof of Proposition~\ref{prop:3} makes use of the following minimization
result. Define for $v\in H^1(\R)$ the Nehari functional corresponding to
\eqref{eq:profile} by 
\[
I(v)=I_{\omega,c}(v)=\norm{v_x}_{L^2}^2+\Im\int_{\R}|v|^2\bar v v_xdx+\omega \norm{v}_{L^2}^2+c\Im\int_{\R} v\bar v_xdx.
\]
Remark that $I(v)=\frac{\partial }{\partial t}S(tv)_{|t=0}$. Define
the minimum of
the action $S$ on the Nehari manifold by
\[
m_{\mathcal N}=\inf\{S(v);\, v\in H^1(\R)\setminus\{0\},\,I(v)=0\}
\]
and let 
the set of minimizers be denoted by
\[
\mathcal G_{\mathcal N}:=\{ v\in H^1(\R)
\setminus\{0\};\,I(v)=0,\,S(v)=m_{\mathcal N}\}. 
\]

\begin{proposition}
  \label{prop:nehari}
  Let  $\omega, c\in\R$ be such that $4\omega>c^2$. Then $m_{\mathcal
    N}>0$ and
  $\phi_{\omega,c}$ is up to phase shift and translation the unique
  minimizer for $m_{\mathcal N}$, that is
  \[
  \mathcal G_{\mathcal N}=\{e^{i\theta}\phi_{\omega,c}(x-y);\,\theta,y\in\R\}.
  \]
\end{proposition}

Proposition~\ref{prop:nehari} was first obtained by Colin and Ohta
\cite[Lemma 10]{CoOh06}. For the sake of completeness, we reproduce
here the proof. 

We will make use of the following two classical lemmas (see~\cite{BrLi83,FrLiLo86,Li83}). 

\begin{lemma}\label{lem:lieb}
  Let $(v_n)$ be a bounded sequence in $H^1(\R)$. Assume that there
  exists $p\in(2,\infty)$ such that
  \[
  \limsup_{n\to\infty}\norm{v_n}_{L^p}>0.
  \]
  Then there exist $(y_n)\subset \R$ and $v_\infty\in
  H^1(\R)\setminus\{0\}$ such that $(v_n(\cdot-y_n))$ has a convergent
  subsequence to $v_\infty$ weakly in $H^1(\R)$.
\end{lemma}

\begin{lemma}\label{lem:brezis-lieb}
  Let $2\leq p < \infty$ and $(v_n)$ be a bounded sequence in
  $L^p(\R)$. Assume that $v_n\to v_\infty$ a.e. in $\R$. Then 
  \[
  \norm{v_n}_{L^p}^p-\norm{v_n-v_\infty}_{L^p}^p-\norm{v_\infty}_{L^p}^p\to
  0,\quad\text{ as }n\to\infty.
  \]
\end{lemma}

\begin{proof}[Proof of Proposition~\ref{prop:nehari}]
  \emph{Step 1. We show that $m_{\mathcal
      N}>0$.} Let $v\in H^1(\R)\setminus\{0\}$ be such that
  $I(v)=0$. Introduce the notation 
  \[
  L(v)=L_{\omega,c}(v)=\norm{v_x}_{L^2}^2+\omega
  \norm{v}_{L^2}^2+c\Im\int_{\R} v\bar
  v_xdx
  \]
  and remark that
  \[
  L(v)
  =\norm*{v_x-\frac{ic}{2}v}_{L^2}^2+\left(\omega-\frac
    {c^2}{4}\right)\norm{v}_{L^2}^2.
  \]
  Using $I(v)=0$, we have
  \[
  \norm{v}_{H^1}^2\lesssim \norm*{v_x-\frac{ic}{2}v}_{L^2}^2+\left(\omega-\frac
    {c^2}{4}\right)\norm{v}_{L^2}^2=-\Im\int_{\R}|v|^2\bar v v_xdx\lesssim \norm{v}_{H^1}^4.
  \]
  Therefore there exists $\delta>0$ independant of $v$ such that 
  \[
  \norm{v}_{H^1}^2>\delta.
  \]
  In addition, we have
  \[
  S(v)=S(v)-\frac14I(v)=\frac14L(v)\gtrsim \norm{v}_{H^1}^2>\delta>0.
  \]
  Therefore 
  \[
  m_{\mathcal N}>0.
  \]

  \emph{Step 2. We show that }
  \[
  m_{\mathcal N}=\frac14\inf\{L(v);\, v\in
  H^1(\R)\setminus\{0\},\,I(v)\leq 0\}.
  \]
  Indeed, let $v\in
  H^1(\R)\setminus\{0\}$ be such that $I(v)< 0$. Then there exists
  $\lambda\in(0,1)$ such that $I(\lambda v)=0$. Moreover, 
  \[
  m_{\mathcal N}\leq S(\lambda v)=S(\lambda v)-\frac14 I(\lambda v)= \frac 14 L(\lambda v)=\frac{\lambda^2}{4}L(v)<\frac 14 L(v).
  \]

  \emph{Step 3. We show convergence of the minimizing sequences.} 
  Let $(v_n)\subset H^1(\R)\setminus\{0\}$ be such that
  $I(v_n)=0$ for all $n\in\mathbb N$ and $S(v_n)\to m_{\mathcal N}$ as
  $n\to\infty$.
  In the sequel, all statements will be true up to the extraction of a
  subsequence. 
  The sequence $(v_n)$ is bounded from above and below in $H^1(\R)$. Moreover, we claim that 
  \[
  \limsup_{n\to\infty}\norm{v_n}_{L^6}>0. 
  \]
  Indeed, assume by contradiction that
  $\lim_{n\to\infty}\norm{v_n}_{L^6}=0$. Then from $I(v_n)=0$ and by
  Cauchy-Schwartz inequality, we have
  \[
  \norm{v_n}_{H^1}^2\lesssim L(v_n)
  =\abs*{-\Im\int_{\R}|v_n|^2\bar v_n\partial_x
    v_ndx}\lesssim \norm{v_n}_{L^6}^3\norm{\partial_xv_n}_{L^2}\to 0.
  \]
  It is a contradiction with the boundedness from below of $(v_n)$ in
  $H^1(\R)$. Therefore $\limsup_{n\to\infty}\norm{v_n}_{L^6}>0$ and we
  can apply Lemma~\ref{lem:lieb} to obtain the existence of
  $v_\infty\in H^1(\R)\setminus\{0\}$ and $(y_n)\subset \R$ such that
  \[
  v_n(\cdot-y_n)\rightharpoonup v_\infty \text{ weakly in }H^1(\R),\quad
  v_n(\cdot-y_n)\to v_\infty\text{ a.e.}
  \]
  From now on, we replace $v_n$ by $v_n(\cdot-y_n)$. 
  By weak convergence we have 
  \begin{equation}
    \label{eq:ll}
    L(v_n)-L(v_n-v_\infty)-L(v_\infty)\to0.
  \end{equation}
  By Lemma
~\ref{lem:brezis-lieb} we have
  \begin{align}
    \norm{v_n}_{L^4}^4-\norm{v_n-v_\infty}_{L^4}^4-\norm{v_\infty}_{L^4}^4&\to0,\label{eq:conv1}\\
    \norm{v_n}_{L^6}^6-\norm{v_n-v_\infty}_{L^6}^6-\norm{v_\infty}_{L^6}^6&\to 0.\label{eq:conv2}
  \end{align}
  Remark that for any $v\in H^1(\R)$ we can rewrite $I(v)$ as
  \begin{equation}\label{eq:conv3}
    I(v)=\norm*{v_x-\frac{ic}{2}v+\frac{i}{2}|v|^2v}_{L^2}^2+\left(\omega-\frac
      {c^2}{4}\right)\norm{v}_{L^2}^2+\frac{c}{2}\norm{v}_{L^4}^4-\frac{1}{4}\norm{v}_{L^6}^6.
  \end{equation}
  Introduce the functions $w_n$ and $w_\infty$ defined by
  \[
  w_n=\partial_xv_n-\frac{ic}{2}v_n+\frac{i}{2}|v_n|^2v_n,\quad
  w_\infty=\partial_xv_\infty-\frac{ic}{2}v_\infty+\frac{i}{2}|v_\infty|^2v_\infty.
  \]
  Then $w_n\rightharpoonup w_\infty$ in $L^2(\R)$ and we have
  \[
  \norm{w_n}_{L^2}^2-\norm{w_n-w_\infty}_{L^2}^2-\norm{w_\infty}_{L^2}^2\to 0.
  \]
  Combined with~\eqref{eq:conv1},\eqref{eq:conv2} and~\eqref{eq:conv3},
  this gives
  \begin{equation}\label{eq:conv4}
    I(v_n)-I(v_n-v_\infty)-I(v_\infty)\to 0.
  \end{equation}

  \emph{Step 4. We show that the minimal element $v_\infty$ verifies $I(v_\infty)\leq 0$.} Assume by contradiction that
  $I(v_\infty)> 0$. From~\eqref{eq:conv4}
  and since $I(v_n)=0$, we have
  \[
  \lim_{n\to\infty}I(v_n-v_\infty)=-I(v_\infty)<0.
  \]
  Therefore, by Step 2, $L(v_n-v_\infty)>4m_{\mathcal N}$. Combining this with
  $\lim_{n\to\infty}L(v_n)=4m_{\mathcal N}$ and~\eqref{eq:ll}, we obtain
  \[
  L(v_\infty)=\lim_{n\to\infty}(L(v_n)-L(v_n-v_\infty))\leq 0.
  \]
  However $v_\infty\neq 0$ and thus $L(v_\infty)>0$, which is a
  contradiction. Therefore $I(v_\infty)\leq 0$. 

  \emph{Step 5. We prove that $v_\infty$ achieves the minimum for
    $m_{\mathcal N}$.} 
  We have
  \[
  4m_{\mathcal N}\leq L(v_\infty)\leq
  \lim_{n\to\infty}L(v_n)=4m_{\mathcal N}.
  \]
  Therefore 
  \[
  L(v_\infty)=
  \lim_{n\to\infty}L(v_n)
  \]
  and thus $v_n\to v_\infty$ strongly in $H^1(\R)$. This implies
  $I(v_\infty)=0$ and therefore 
  \[
  S(v_\infty)=m_{\mathcal N}.
  \]

  \emph{Step 6. Conclusion.} The last step of the proof consists in proving that $v_\infty$ is in
  fact a solution of~\eqref{eq:profile}. As $v_\infty$ is a minimizer
  for $m_{\mathcal N}$, there exists a Lagrange multiplier $\lambda\in\R$
  such that 
  \[
  S'(v_\infty)=\lambda I'(v_\infty). 
  \]
  Therefore
  \[
  0=I(v_\infty)=\dual{S'(v_\infty)}{v_\infty}=\lambda \dual{I'(v_\infty)}{v_\infty}.
  \]
  Since 
  \[
  \dual{I'(v_\infty)}{v_\infty}=2L(v_\infty)+4\Im\int|v_\infty|^2\bar
  v_\infty\partial_x v_\infty dx=-2L(v_\infty)<0,
  \]
  we necessarily have $\lambda=0$. This implies that $v_\infty$ is a
  solution of~\eqref{eq:profile}. Since by uniqueness any solution of
~\eqref{eq:profile}  can be written for some $\theta,y\in\R$ as 
  \[
  v_\infty=e^{i\theta}\phi_{\omega,c}(\cdot-y),
  \]
  this concludes the proof of Proposition~\ref{prop:nehari}. 
\end{proof}

With Proposition~\ref{prop:nehari} in hand, we can now proceed to the
proof of Proposition~\ref{prop:3}.

\begin{proof}[Proof of Proposition~\ref{prop:3}]
  For simplicity in notation we drop the subscript ${\omega,c}$ in
  $\phi_{\omega,c}$ and simply write $\phi$ instead. We shall follow
  more or less the scheme of proof already used in~\cite{BeGhLe14}. 
  We start by rewriting $S''(\phi)$ (see~\eqref{eq:explicit}) as a two by two matrix
  operator acting on $(\Re(\eps),\Im(\eps))$:
  \[
  S''(\phi)=
  \begin{pmatrix}
    -\partial_{xx}+\omega+2\Re(\phi)\Im(\phi_x)&2\Im(\phi)\Im(\phi_x)+|\phi|^2\partial_x-c\partial_x\\
    -2\Re(\phi)\Re(\phi_x)-|\phi|^2\partial_x+c\partial_x&-\partial_{xx}+\omega-2\Im(\phi)\Re(\phi_x)
  \end{pmatrix}.
  \]

  \emph{Step 1. Global spectral picture.} 
  Since $\phi$ is exponentially localized, $S''(\phi)$ can be considered
  as compact perturbation of 
  \[
  \begin{pmatrix}
    -\partial_{xx}+\omega&-c\partial_x\\
    c\partial_x&-\partial_{xx}+\omega
  \end{pmatrix}.
  \]
  Therefore its essential spectrum is $[\omega-\frac{c^2}{4},\infty)$ and
  by Weyl's Theorem its spectrum in $(-\infty,\omega-\frac{c^2}{4})$ consists of isolated
  eigenvalues. Due to the variational characterization Proposition
~\ref{prop:nehari}, $S''(\phi)$ admits at most one negative
  eigenvalue. Using that $\phi$ satisfies to the Nehari constraint
  $I(\phi)=0$, we have
  \[
  \dual{S''(\phi)\phi}{\phi}=2\Im\int_{\R}|\phi|^2\bar \phi
  \phi_xdx=-2L(\phi)<0.
  \]
  This implies that
  the operator $S''(\phi)$ has exactly one negative eigenvalue. 

  \emph{Step 2. Non-degeneracy of the kernel.} 
  We claim that $\ker(S''(\phi))=\spn\{ i\phi,\phi_x \}$.
  Write 
  \[
  \eps=\exp\left(i\left(\frac
      c2x-\frac14\int_{-\infty}^x|\phi|^2dy\right)\right)\left(k-\frac{i\varphi}{2}\int_{-\infty}^x\varphi\Re(k)dy\right),
  \]
  where $\varphi=\varphi_{\omega,c}$ is the real part of $\phi$ given
  explicitly in formula~\eqref{eq:varphi-formula}.
  Then 
  \begin{multline*}
    S''(\phi)\eps=-k_{xx}+\left(\omega-\frac{c^2}{4}\right) k
    +i \frac12\varphi\varphi_x \Im(k)
    -i\frac12\varphi^2\Im(k_x)\\
    +\frac
    c2\varphi^2k
    +c\varphi^2\Re(k)
    -\frac{3}{16}\varphi^4k
    -\frac{12}{16}\varphi^4\Re(k).
  \end{multline*}
  Separating in real and imaginary part, we obtain
  \[
  S''(\phi)\eps =L_+\Re(k)+iL_-\Im(k),
  \]
  where 
  \begin{align*}
    L_+&=-\partial_{xx}+\left(\omega-\frac{c^2}{4}\right)
         +\frac{3c}{2}\varphi^2-\frac{15}{16}\varphi^4,\\
    L_-&=-\partial_{xx}+\left(\omega-\frac{c^2}{4}\right)
         +\frac{c}{2}\varphi^2-\frac{3}{16}\varphi^4+\frac{\varphi\varphi_x}{2}-\frac{\varphi^2}{2}\partial_x.
  \end{align*}
  Hence proving non-degeneracy for $S''(\phi)$ amounts to proving
  non-degeneracy of $L_+$ and $L_-$. That is, we want to prove that
  \[
  \ker(L_+)=\spn\{\varphi_x\},\quad \ker(L_-)=\spn\{\varphi\}.
  \]
  Since $\varphi$ satisfies~\eqref{eq:varphi}, it is clear that $L_-\varphi=0$. Let $v\in H^2(\R)\setminus\{0\}$ be
  such that $L_-v=0$. Consider the Wronskian of $\varphi$ and $v$:
  \[
  W=\varphi_xv-\varphi v_x.
  \]
  It verifies the equation
  \[
  W'=-\frac{\varphi^2}{2}W
  \] 
  and therefore it is of the form
  \[
  W(x)=Ce^{-\frac12\int_0^x\varphi^2(y)dy}.
  \]
  Since $\varphi,v\in H^2(\R)$, we have
  \[
  \lim_{x\to\pm\infty} \varphi(x)=
  \lim_{x\to\pm\infty}  \varphi_x(x)=
  \lim_{x\to\pm\infty} v(x)=
  \lim_{x\to\pm\infty} v_x(x)=0.
  \]
  Therefore, 
  \[
  \lim_{x\to\pm\infty}W(x)=0,
  \]
  which is possible only if $W\equiv0$. Therefore $v\in\spn\{\varphi\}$
  and this proves non-degeneracy of $L_-$. The non-degeneracy of $L_+$
  follows from similar arguments.

  \emph{Step 3. Construction of a negative direction.} 
  Differentiating~\eqref{eq:profile} with respect to $\omega$ and $c$,
  we observe that
  \[
  S''(\phi)\partial_\omega\phi=-\phi,\quad S''(\phi)\partial_c\phi=-i\phi_x.
  \] 
  Let $\mu\in\R$ to be chosen later and define
  \[
  \psi=\partial_\omega\phi+\mu \partial_c\phi.
  \]
  We have
  \[
  \dual{S''(\phi)\psi}{\psi}=
  -\dual{\phi}{\partial_\omega\phi}
  -\mu\dual{\phi}{\partial_c\phi}
  -\mu\dual{i\phi_x}{\partial_\omega\phi}
  -\mu^2\dual{i\phi_x}{\partial_c\phi}.
  \]
  Moreover, using~\eqref{FM:Mass-dw}-\eqref{FM:Mom-c}, we get
  \begin{align*}
    \dual{\phi}{\partial_\omega\phi}&=\frac{\partial}{\partial\omega}M(\phi_{\omega,c})=\frac{-\frac{c}{\omega}}{\sqrt{4\omega-c^2}},&
                                                                                                                                       \dual{\phi}{\partial_c\phi}&=\frac{\partial}{\partial c}M(\phi_{\omega,c})=\frac{2}{\sqrt{4\omega-c^2}},\\
    \dual{i\phi_x}{\partial_\omega\phi}&=\frac{\partial}{\partial \omega}P(\phi_{\omega,c})=\frac{2}{\sqrt{4\omega-c^2}},&
                                                                                                                           \dual{i\phi_x}{\partial_c\phi}&=\frac{\partial}{\partial
                                                                                                                                                           c}P(\phi_{\omega,c})=\frac{-c}{\sqrt{4\omega-c^2}}.
  \end{align*}
  This gives
  \begin{equation}\label{eq:recall}
    \dual{S''(\phi)\psi}{\psi}=\left( \frac{c}{\omega}-4\mu+c\mu^2 \right) \frac{1}{\sqrt{4\omega-c^2}}.
  \end{equation}
  Therefore, since $4\omega-c^2>0$, there always exists $\mu$ such that 
  \[
  \dual{S''(\phi)\psi}{\psi}<0.
  \]
  Let such a $\mu$ be fixed now. 
  If $c<0$, we can choose $\mu=0$. If $c= 0$, we can choose
  $\mu=1$ and if $c>0$ we can choose $\mu=\frac{2}{c}$. 

  \emph{Step 4. Positivity.} 
  Let us now denote by $-\lambda$ and $\xi$ the negative eigenvalue of
  $S''(\phi)$ and its corresponding normalized eigenvector, i.e.
  \[
  S''(\phi)\xi=-\lambda\xi,\quad\norm{\xi}_{L^2}=1.
  \]
  We write the decomposition of $\psi$ along the spectrum of
  $S''(\phi)$:
  \[
  \psi=\alpha \xi+\zeta+\eta,
  \]
  with $\alpha\in\R$, $\alpha\neq0$, $\zeta\in\ker(S''(\phi))$ and $\eta$ in the positive eigenspace
  of $S''(\phi)$. In particular, we have
  \[
  \dual{S''(\phi)\eta}{\eta}\gtrsim \norm{\eta}_{L^2}^2.
  \]
  Take $\eps\in H^1(\R)\setminus\{0\}$ such that the following orthogonality
  conditions hold
  \[
  \psld{\eps}{i\phi}=
  \psld{\eps}{\partial_x\phi}=
  \psld{\eps}{\phi+i\mu\phi_x}=
  0.
  \]
  We also write the  decomposition of $\eps$ along the spectrum of
  $S''(\phi)$:
  \[
  \eps=\beta \xi+\rho,
  \]
  with $\rho$ in the positive eigenspace of $S''(\phi)$. Since 
  \[
  -\phi-i\mu \phi_x=S''(\phi)\psi= -\alpha \lambda\xi+S''(\phi)\eta ,
  \] 
  we have
  \[
  0=-\psld{\phi+i\mu \phi_x}{\eps}=\dual{S''(\phi)\psi}{\eps}=-\alpha\beta\lambda+\dual{S''(\phi)\eta}{\rho},
  \]
  thus
  \[
  \dual{S''(\phi)\eta}{\rho}=\alpha\beta\lambda.
  \]
  From Cauchy-Schwartz inequality, we have
  \[
  (\alpha\beta\lambda)^2=\dual{S''(\phi)\eta}{\rho}^2\leq \dual{S''(\phi)\eta}{\eta}\dual{S''(\phi)\rho}{\rho}.
  \]
  In addition, since $\dual{S''(\phi)\psi}{\psi}<0$, we have
  \[
  \dual{S''(\phi)\eta}{\eta}<\alpha^2 \lambda.
  \]
  Therefore, 
  \[
  \dual{S''(\phi)\eps}{\eps}=-\beta^2\lambda
  +\dual{S''(\phi)\rho}{\rho}\geq -\beta^2\lambda+\frac{(\alpha\beta\lambda)^2}{\dual{S''(\phi)\eta}{\eta}}>-\beta^2\lambda+\frac{(\alpha\beta\lambda)^2}{\alpha^2\lambda}=0.
  \]

  \emph{Step 5. Coercivity.} 
  To obtain the desired coercivity property, we argue by
  contradiction. Let $\eps_n$ be such that $\norm{\eps_n}_{H^1}=1$ and 
  \[
  \lim_{n\to\infty}\dual{S''(\phi)\eps_n}{\eps_n}=0.
  \]
  By boundedness, there exists $\eps_\infty\in H^1(\R)$ such that 
  \[
  \eps_n\rightharpoonup  \eps_{\infty} \text{ weakly in }H^1(\R). 
  \]
  On one hand, by weak convergence, $\eps_\infty$ verifies the orthogonality
  conditions
  \[
  \psld{\eps_\infty}{i\phi_{\omega,c}}=
  \psld{\eps_\infty}{\partial_x\phi_{\omega,c}}=
  \psld{\eps_\infty}{\phi_{\omega,c}+i\mu\phi_x}=
  0.
  \]
  In particular, if $\eps_\infty \not \equiv 0$, then, by Step 4, we have
  \[
  \dual{S''(\phi)\eps_\infty}{\eps_\infty}>0.
  \]
  On the other hand, we remark that $\dual{S''(\phi)\eps}{\eps} =H(\eps)$ can be written
  \[
  H(\eps)=\norm*{\eps_x-i\frac c2\eps}_{L^2}^2+\left(\omega-\frac{c^2}{4}\right)\norm{\eps}_{L^2}^2+\Im\int_\R|\phi|^2\bar\eps
  \eps_xdx+2\int_\R\Re(\phi \bar \eps)\Im(\phi_x \bar\eps)dx.
  \]
  By weak convergence of $\eps_n$, exponential localization of $\phi$ and $\phi_x$ and
  compactness of the injection of $H^1$ into $L^2$ for
  bounded domain, we
  have 
  \[
  \dual{S''(\phi)\eps_\infty}{\eps_\infty}\leq \lim_{n\to\infty}\dual{S''(\phi)\eps_n}{\eps_n}=0.
  \]
  Therefore we must have $\eps_\infty=0$. Since
  $\norm{\eps_n}_{H^1}=1$, there exists $\delta>0$ such that
  \[
  \norm*{\partial_x\eps_n-i\frac c2\eps_n}_{L^2}^2+\left(\omega-\frac{c^2}{4}\right)\norm{\eps_n}_{L^2}^2>\delta.
  \]
  Moreover, since $\lim_{n\to\infty}\dual{S''(\phi)\eps_n}{\eps_n}=0$,
  we have
  \[
  \Im\int_\R|\phi|^2\eps_\infty
  (\bar\eps_\infty)_xdx+2\int_\R\Re(\phi \bar \eps_\infty)\Im(\bar\phi_x \eps_\infty)dx<-\delta<0.
  \]
  However, it is a contradiction with $\eps_\infty=0$. Hence the
  coercivity result Proposition~\ref{prop:3} holds. 
\end{proof}

\section{The Bootstrap Argument}
\label{sec:bootstrap-argument}

This section is devoted to the proof of Theorem~\ref{thm:1} using a
bootstrap argument which will be proved later.

Let $c_1<\cdots<c_N$ and $0<\omega_1<\dots<\omega_N$ be such
that $c_j^2<4\omega_j$
for
$j=1,\dots,N$ and the speed-frequency ratio assumption
\eqref{eq:speed-frequency-ratios} is verified.
Let $(\phi_j)=(\phi_{\omega_j,c_j})$ be the corresponding solitons profiles.
We define the minimal decay rate of the profiles by
\begin{equation}
  \label{eq:def-omega}
  \omega_\star=\min\left\{ \sqrt{\omega_j-\frac{c_j^2}{4}};\;j=1,\dots,N \right\}.
\end{equation}
Define also the \emph{minimal relative speed} $c_\star$ by
\begin{equation}\label{eq:c_star}
  c_\star=\min\left\{ |c_j-c_k|;\;j,k=1,\dots,N,\,j\neq k\right\}.
\end{equation}

Given $A_0,L,\delta>0$, define a tubular neighborhood of the
$N$-soliton profiles by 
\begin{multline*}
  \mathcal V(\delta,L,A_0)\\=\left\{ u\in
    H^1(\R);\;\inf_{\substack{x_j>x_{j-1}+L\\\theta_j\in\R}}\norm[\bigg]{u-\sum_{j=1}^Ne^{i\theta_j}\phi_j(\cdot-x_j
      )}_{H^1}<A_0\left(\delta+e^{-\frac{1}{32}\omega_\star L}\right)\right\}.
\end{multline*}

Theorem~\ref{thm:1} is a straightforward consequence of the following
bootstrap result.

\begin{proposition}
  [Bootstrap]\label{prop:bootstrap}
  There exist $A_0>1$, fixed, and $L_0>0$ and $\delta_0>0$ such that for all
  $L>L_0$, $0<\delta<\delta_0$ the following property is satisfied. 
  If $u_0\in
  H^1(\R)$ verifies
  \[
  u_0\in \mathcal V(\delta,L,1),
  \]
  and if $t^\star>0$ is such that for all $t\in[0,t^\star]$ the
  solution $u$ of~\eqref{eq:nls} with $u(0)=u_0$ verifies
  \[
  u(t)\in \mathcal V(\delta,L,A_0),
  \]
  then for all $t\in[0,t^\star]$ we have
  \[
  u(t)\in \mathcal V\left(\delta,L,\frac{A_0}{2}\right).
  \]
\end{proposition}

Being performing the proof of Proposition~\ref{prop:bootstrap}, let us
indicate how it implies Theorem~\ref{thm:1}.

\begin{proof}[Proof of Theorem~\ref{thm:1}]
  Since $u_0\in\mathcal V(\delta,L,1)$, and $u$ is continuous in
  $H^1(\R)$, there exists a maximal time
  $t^\star\in(0,\infty]$ such that for all $t\in[0,t^\star)$ we have 
  \[
  u(t)\in\mathcal V(\delta,L,A_0). 
  \] 
  Arguing by contradiction, we assume that $t^\star<\infty$. By
  Proposition~\ref{prop:bootstrap}, for all $t\in[0,t^\star)$
  we have 
  \[
  u(t)\in\mathcal V\left(\delta,L,\frac{A_0}{2}\right). 
  \] 
  By continuity of $u$ in $H^1(\R)$, there must exist
  $t^{\star\star}>t^\star$ such that for all $t\in[0,t^{\star\star})$ we have 
  \[
  u(t)\in\mathcal V(\delta,L,A_0). 
  \] 
  This however contradicts the maximality of $t^\star$. Hence
  $t^\star=\infty$. This concludes the proof.
\end{proof}

The rest of this paper is devoted to the proof of Proposition
\ref{prop:bootstrap}. 
From now on, we assume that we are given $A_0>1$, $L>L_0=L_0(A_0)>0$ and
$0<\delta<\delta_0=\delta_0(A_0)$ such that 
\[
u_0\in \mathcal V(\delta,L,1),
\]
and there exists $t^\star>0$ such that for all $t\in[0,t^\star]$ the
solution $u$ of~\eqref{eq:nls} with $u(0)=u_0$ verifies
\[
u(t)\in \mathcal V(\delta,L,A_0).
\]
In the sequel, we shall always assume  $t\in[0,t^\star]$.

\section{Modulation}
\label{sec:modulation}

We first explain how to decompose $u$ close to the sum of
solitons. Roughly speaking, we project $u$ on the manifold of the sum
of soliton
profiles modulated in phase, speed, space and scaling. Since we
impose the modulated speed and scaling to have the same ratio as the
original speed and scaling, we in fact modulate on a family of $3N$
parameters.

\begin{proposition}
  [Modulation]\label{prop:modulation}
  For $\delta$ and $1/L$ small enough, the following property is
  verified.
  For $j=1,\dots, N$ there exist (unique) $\mathcal C^1$-functions
  \[
  \tilde\theta_j:[0,t^\star]\to\R,\quad \tilde \omega_j:[0,t^\star]\to(0,\infty),\quad
  \tilde x_j:[0,t^\star]\to\R, \quad   \tilde c_j:[0,t^\star]\to\R
  \]
  such that if we define modulated solitons $\tilde R_j$ and $\eps$ by
  \[
  \tilde R_j(t)=e^{i\tilde\theta(t)}\phi_{\tilde\omega_j(t),\tilde c_j(t)}(\cdot-\tilde
  x_j(t)),
  \quad \eps(t)=u(t)-\sum_{j=1}^N\tilde R_j(t),
  \]
  then $\eps$ satisfies for all $t\in [0,t^\star]$ the orthogonality
  conditions (the constants $\mu_j$ are given by Proposition~\ref{prop:3})
  \begin{equation}
    \label{eq:ortho}
    \psld{\eps}{i\tilde R_{j}}
    =\psld{\eps}{\partial_x\tilde R_{j}}=
    \psld{\eps}{\tilde R_j+\mu_ji\partial_x\tilde R_j}=
    0,\quad j=1,\dots,N.
  \end{equation}
  The scaling and speed parameters verify for all $t\in[0,t^\star]$ and for any $j=1,\dots,N$ the
  relationship
  \begin{equation}\label{eq:wc}
    \tilde c_j(t)-c_j=\mu_j(\tilde \omega_j(t)-\omega_j).
  \end{equation}
  Moreover, there exists $\tilde C>0$ such that for all $t\in [0,t^\star]$ we have
  \begin{gather}
    \label{eq:control}
    \norm{\eps(t)}_{H^1}+\sum_{j=1}^N\left(|\tilde\omega_j(t)-\omega_j|+|\tilde
      c_j(t)-c_j|\right)\leq \tilde CA_0\left(
      \delta+e^{-\frac{1}{4}
        \omega_\star L}\right), \\ 
    \tilde
    x_{j+1}(t)-\tilde x_j(t)>\frac {L}{2},\quad\text{for }j=1,\dots,N-1,\label{eq:dist-mod-x}
  \end{gather}
  and the derivatives in time verify
  \begin{equation}\label{eq:dynamical}
    \sum_{j=1}^N   |\partial_t\tilde c_j|+
    |\partial_t\tilde\omega_j|+\left|\partial_t\tilde\theta_j-\tilde\omega_j\right|+\left|\partial_t\tilde
      x_j-\tilde c_j\right|
    \leq \tilde C
    \norm{\eps(t)}_{H^1} 
    +\tilde Ce^{-\frac{1}{4}
      \omega_\star L} .
  \end{equation}
  Finally, at $t=0$ the estimate does not depend on $A_0$ and we have
  \begin{equation}
    \label{eq:initial-time}
    \norm{\eps(0)}_{H^1}+\sum_{j=1}^N
    \left(|\tilde\omega_j(0)-\omega_j|+|\tilde
      c_j(0)-c_j|\right)
    \leq \tilde C\left(
      \delta+e^{-\frac{1}{4}
        \omega_\star L}\right).
  \end{equation}
\end{proposition}

\begin{remark}
  We modulate here in a different way as in
~\cite{MaMeTs06}. First, we (artificially) modulate also in speed,
  whereas in~\cite{MaMeTs06} modulation was only on phase, position
  and scaling. Second, we modulate in position on the full profile
  $\phi_{\omega,c}$, whereas in~\cite{MaMeTs06} modulation in position
  was done only on the modulus of the profile (the equivalent of
  $\varphi_{\omega,c}$ in our setting). The way we modulate is more
  natural, but it introduces a technical difficulty in the proof of
  the modulation result. Precisely, the Jacobian given by 
~\eqref{eq:jacobian} is not diagonal and its invertibility is not
  obvious. 
  We are able to overcome this difficulty in our setting thanks to our
  knowledge of the explicit expressions of the profiles. 
\end{remark}

\begin{proof}[Proof of Proposition~\ref{prop:modulation}]
  The existence and regularity of the functions $(\tilde\theta_j,\tilde
  x_j,\tilde\omega_j,\tilde c_j)$ follow  from classical arguments
  involving the Implicit Function Theorem. The main difficulty here is that we have a
  non-diagonal Jacobian, and proving its invertibility requires an additional
  argument  compare to the usual setting.  The modulation equations
~\eqref{eq:dynamical} are obtained via the combination of the equation
  verified by $\eps$ with the orthogonality conditions
~\eqref{eq:ortho}. We only give the important
  steps of the proof.

  Let
  \[
  q=\left(\tilde \theta_1,\dots,\tilde\theta_N; \tilde x_1,\dots,
    \tilde x_N; \tilde \omega_1, \dots, \tilde \omega_N; v \right)\in
  \R^{3N}\times H^1(\R)
  \]
  and
  \[
  q_0=\bigg( \theta_1,\dots,\theta_N;  x_1,\dots,  x_N;  \omega_1, \dots,  \omega_N; \sum_{j=1}^Ne^{i\theta_j}\phi_j(\cdot-x_j
  ) \bigg).
  \]
  For a given $q$, define the speeds $\tilde c_j$ by
  \begin{equation}\label{eq:wc-pf}
    \tilde c_j=c_j+\mu_j(\tilde\omega_j-\omega_j).
  \end{equation}
  Define also
  the modulated profiles $\tilde R_j$ and the difference $\eps$ by
  \[
  \tilde R_j=e^{i\tilde\theta}\phi_{\tilde\omega_j,\tilde c_j}(\cdot-\tilde
  x_j),
  \quad \eps=v-\sum_{j=1}^N\tilde R_j,
  \]
  Consider the function $\Phi:\R^{3N}\times H^1(\R)\to \R^{3N}$ defined
  by $\Phi=(\Phi^1,\Phi^2,\Phi^3)$ with 
  \[
  \Phi_j^1=\psld{\eps}{i\tilde R_{j}},\quad
  \Phi_j^2=\psld{\eps}{\partial_x\tilde R_{j}},\quad
  \Phi_j^3=\psld{\eps}{\tilde R_j+i\mu_{j}\partial_x\tilde R_j}.
  \]
  Note that
  \[
  \frac{\partial \eps}{\partial \tilde\theta_j}\Big|_{q=q_0}=-i R_{j},\quad
  \frac{\partial \eps}{\partial \tilde x_j}\Big|_{q=q_0}=\partial_xR_{j},
  \]
  and by the relationship~\eqref{eq:wc-pf}, we have
  \[
  \frac{\partial \eps}{\partial \tilde
    \omega_j}\Big|_{q=q_0}=-\left(\partial_\omega R_j+\mu_j\partial_c R_j\right).
  \]
  Then, 
  \begin{align*}
    \frac{d\Phi_j^1}{d \tilde\theta_j}\Big|_{q=q_0}&=-\norm{R_j}_{L^2}^2=-8 \arctan \left( \sqrt{\frac{2\sqrt{\omega_j}+c_j}{2\sqrt{\omega_j}-c_j}}
                                                     \right) 
                                                     ,\\
    \frac{d\Phi_j^2}{d \tilde x_j}\Big|_{q=q_0}&=\norm{\partial_x
                                                 R_j}_{L^2}^2=
                                                 8\omega_j \arctan \left( \sqrt{\frac{2\sqrt{\omega_j}+c_j}{2\sqrt{\omega_j}-c_j}}
                                                 \right) 
                                                 ,
  \end{align*}
  where the explicit values come from the formulas in~\eqref{FM:L2} and
~\eqref{eq:norm-derivative-phi}. 
  Using 
  \[
  R_j+i\mu_j\partial_x
  R_j=M'(R_j)+\mu_jP'(R_j),
  \]
  and~\eqref{FM:Mass-dw}--\eqref{FM:Mom-c}, we have
  \begin{align*}
    \frac{d\Phi_j^3}{d\tilde \omega_j}\Big|_{q=q_0}&=-\left(\partial_{ \omega} R_{j}+\mu_{j}\partial_{c} R_{j}, R_j+i\mu_{j}\partial_xR_j\right)_2\\
                                                   &=-\left(\partial_{\omega}M( R_j)+\mu_{j}\left(\partial_{c}M( R_j)+\partial_{ \omega}P( R_j)\right)+\mu_{j}^2\partial_{c}P( R_j)\right)\\
                                                   &=\left( \frac{ c_j}{ \omega_j}-4\mu_j+ c_j\mu_j^2
                                                     \right) \frac{1}{\sqrt{4 \omega_j- c_j^2}}<0 ,
  \end{align*}
  where at the last inequality we recalled~\eqref{eq:recall} and the
  choice of $\mu_j$. 
  Moreover, since $x_{j+1}-x_j>L$, when  $j\neq k$, for $l=1,2,3$ we have
  \begin{align*}
    \frac{d\Phi_j^l}{d \tilde\theta_k}\Big|_{q=q_0}=\frac{d\Phi_j^l}{d
    \tilde x_k}\Big|_{q=q_0}=\frac{d\Phi_j^l}{d\tilde
    \omega_k}\Big|_{q=q_0}=O(e^{-\frac12\omega_\star L}).
  \end{align*}
  We easily verify that
  \begin{align*}
    \left(\frac{d\Phi_j^3}{d \tilde\theta_j}\Big|_{q=q_0}\right)&=\psld{-i R_{j}}{ R_{j}+i\mu_j\partial_x R_j}=0,\\
    \left(\frac{d\Phi_j^3}{d
    \tilde x_j}\Big|_{q=q_0}\right)
                                                                &=\psld{\partial_x R_j}{ R_{j}+i\mu_j\partial_x R_j}
                                                                  =0.
  \end{align*}
  Furthermore, using the explicit expression of $\phi_j$ as well as the formulas
~\eqref{FM:L2}--\eqref{FM:L4}, we have
  \begin{multline*}
    \frac{d\Phi_j^1}{d
      \tilde x_j}\Big|_{q=q_0}=
    -\frac{d\Phi_j^2}{d
      \tilde\theta_j}\Big|_{q=q_0}
    =\psld{\partial_x R_{j}}{i
      R_{j}}\\
    =\psld{\partial_x\phi_j}{i
      \phi_j}=\psld{\partial_x\varphi_j+\frac{ic}{2}\varphi_j-\frac{i}{4}|\varphi_j|^2\varphi_j}{i\varphi_j}=\frac
    c2\norm{\varphi_j}_{L^2}^2-\frac14\norm{\varphi_j}_{L^4}^4
    \\=-2\sqrt{4\omega_j-c_j^2}.
  \end{multline*}
  The Jacobian matrix of the derivative of the function $q\mapsto\Phi(q)$ with respect to $(\tilde\theta_j,\tilde x_j,
  \tilde\omega_j)$ is the $3\times3$ block matrix
  \begin{equation}\label{eq:jacobian}
    D\Phi=\left.\begin{pmatrix}
        \frac{d\Phi_j^1}{d \tilde\theta_k}&\frac{d\Phi_j^1}{d
          \tilde x_k}&\frac{d\Phi_j^1}{d
          \tilde\omega_k}\\
        \frac{d\Phi_j^2}{d \tilde\theta_k}&\frac{d\Phi_j^2}{d
          \tilde x_k}&\frac{d\Phi_j^2}{d
          \tilde\omega_k}\\
        \frac{d\Phi_j^3}{d \tilde\theta_k}&\frac{d\Phi_j^3}{d
          \tilde x_k}&\frac{d\Phi_j^3}{d
          \tilde\omega_k}
      \end{pmatrix}\right|_{q=q_0}
    .
  \end{equation}
  Each block is diagonal up to $O(e^{-\frac12\omega_\star L})$. All terms on
  the main diagonal are non-zero and have been explicitly
  computed. The block terms at $(3,1)$ and $(3,2)$ are of order
  $O(e^{-\frac12\omega_\star L})$. Therefore, the determinant of the matrix
  is 
  \begin{multline*}
    \det(D\Phi)=
    \prod_{j=1}^N \frac{\left( \frac{ c_j}{ \omega_j}-4\mu_j+ c_j\mu_j^2
      \right) }{\sqrt{4 \omega_j- c_j^2}} 
    \cdot\\
    \cdot \prod_{j=1}^N\left(-
      64\omega_j\left(\arctan \left( \sqrt{\frac{2\sqrt{\omega_j}+c_j}{2\sqrt{\omega_j}-c_j}}
        \right) \right)^2+ \left(2\sqrt{4\omega_j-c_j^2}\right)^2\right)+O(e^{-\frac12\omega_\star L}).
  \end{multline*}
  Consider the function $f(\omega,c)$ defined by
  \[
  f(\omega,c)=8\sqrt{\omega}\arctan \left( \sqrt{\frac{2\sqrt{\omega}+c}{2\sqrt{\omega}-c}}
  \right) -2\sqrt{4\omega-c^2}.
  \]
  By explicit calculations, we have
  \[
  \partial_cf(\omega,c)=\frac{2(c+2\sqrt{\omega})}{\sqrt{4\omega-c^2}}>0.
  \]
  Moreover, at $c=-2\sqrt{\omega}$ the function starts with $f(\omega,-2\sqrt{\omega})= 0$. This implies
  that $f(\omega,c)>0$ for all $(\omega,c)\in\R^2$ with
  $4\omega>c^2$. As a consequence, for any $j=1,\dots,N$ we have
  \begin{equation}\label{eq:help-det}
    64\omega_j\arctan \left( \sqrt{\frac{2\sqrt{\omega_j}+c_j}{2\sqrt{\omega_j}-c_j}}
    \right) >\left(2\sqrt{4\omega_j-c_j^2}\right)^2,
  \end{equation}
  and we infer that
  \[
  \det(D\Phi)\neq 0.
  \]
  Hence we can apply the Implicit Function Theorem to $\Phi$ to obtain
  the existence of a function $\tilde q_1=(\tilde \theta_j, \tilde x_j, \tilde
  \omega_j)$ from $[0,t^\star]$ to $\R^{3N}$ such that for any
  $t\in[0,t^\star]$ we have
  \[
  \Phi(\tilde q_1(t),u(t))=0. 
  \]
  We refer the reader to~\cite{MaMeTs06} for the regularity of the
  modulation parameters. This concludes the first part of the
  proof.

  We now want to obtain the modulation equations
~\eqref{eq:dynamical}. Recall first that $u$ is a solution of
~\eqref{eq:nls}, hence it verifies 
  \[
  iu_t=E'(u).
  \]
  Since $u=\sum_{j=1}^N\tilde R_j+\eps$, the equation verified by $\eps$
  is
  \[
  i\eps_t+\sum_{j=1}^N\left(-\partial_t\tilde \theta_j\tilde
    R_j-i\partial_t \tilde x_j\partial_x\tilde
    R_j+i\partial_t\tilde\omega_j\left(\partial_\omega\tilde
      R_j+i\mu_j\partial_c\tilde R_j\right)\right)
  =E'\bigg(\sum_{j=1}^N\tilde R_j+\eps\bigg).
  \]
  By exponential localization of the solitons and~\eqref{eq:dist-mod-x}, we have  
  \[
  E'\bigg(\sum_{j=1}^N\tilde R_j+\eps\bigg)=\sum_{j=1}^N \left(E'
    \left(\tilde R_j\right)+E'' \left(\tilde
      R_j\right)\eps\right)+O\left(e^{-\frac{1}{4}\omega_\star L}\right)+ O\left(\norm{\eps}_{H^1}^2\right).
  \]
  Recall that each $\tilde R_j$ verifies the equation
  \[
  E'(\tilde R_j)+\tilde \omega_jM'(R_j)+\tilde c_jP'(R_j)=0.
  \]
  Therefore, the equation for $\eps$ can be written as
  \begin{multline*}
    i\eps_t+\sum_{j=1}^N\left((\tilde \omega_j-\partial_t\tilde \theta_j)\tilde
      R_j+(\tilde c_j-\partial_t \tilde x_j)i\partial_x\tilde
      R_j+i\partial_t\tilde\omega_j\left(\partial_\omega\tilde
        R_j+\mu_j\partial_c\tilde R_j\right)\right)
    \\=
    \sum_{j=1}^N E'' \left(\tilde
      R_j\right)\eps+O\left(e^{-\frac{1}{4}\omega_\star L}\right)+ O\left(\norm{\eps}_{H^1}^2\right).
  \end{multline*}
  One can see already the modulation equations appearing. For
  convenience, we denote by $\Mod(t)$ the vector of modulation
  equations, i.e.
  \[
  \Mod(t)=\left(\tilde \omega_j-\partial_t\tilde \theta_j,\tilde c_j-\partial_t \tilde x_j,i\partial_t\tilde\omega_j\right)_{j=1,\dots,N}.
  \]
  Differentiating with respect to time the orthogonality conditions
~\eqref{eq:ortho}, we obtain
  \begin{align}
    0&=-\psld{i\eps_t}{\tilde R_j}+\psld{\eps}{i\partial_t\tilde R_j},\label{eq:first-orth}\\
    0&=\psld{i\eps_t}{i\partial_x\tilde R_j}+\psld{\eps}{\partial_t\partial_x\tilde R_j},\label{eq:second-orth}\\
    0&=\psld{i\eps_t}{i \tilde R_j-\mu_j\partial_x\tilde
       R_j}+\psld{\eps}{\partial_t(\tilde R_j+i\mu_j\partial_x\tilde R_j)}.\label{eq:third-orth}
  \end{align}
  We have
  \begin{multline*}
    \abs*{\psld{\eps}{i\partial_t\tilde R_j}}
    +\abs*{\psld{\eps}{\partial_t\partial_x\tilde R_j}}
    +\abs*{\psld{\eps}{\partial_t(\tilde R_j+i\mu_j\partial_x\tilde
        R_j)}} 
    \\
    \lesssim \left(1+\abs{\Mod(t)}\right)\norm{\eps}_{L^2}.
  \end{multline*}
  Using the equation for $\eps$, for the first term in the left
  hand side of~\eqref{eq:first-orth}, we get
  \begin{multline*}
    -\psld{i\eps_t}{\tilde
      R_j}=(\tilde\omega_j-\partial_t\tilde\theta_j)\norm{\tilde
      R_j}_{L^2}^2+(\tilde c_j-\partial_t\tilde
    x_j)\psld{i\partial_x\tilde R_j}{\tilde R_j}
    \\+\partial_t\tilde \omega_j\psld{i(\partial_\omega\tilde
      R_j+\mu_j\partial_c\tilde R_j)}{\tilde R_j}
    +(1+\abs{\Mod(t)})O\left(e^{-\frac{1}{4}\omega_\star L}\right)+
    O\left(\norm{\eps}_{H^1}\right).
  \end{multline*}
  For the first term in the left
  hand side of~\eqref{eq:second-orth}, we get
  \begin{multline*}
    -\psld{i\eps_t}{i\partial_x\tilde
      R_j}=(\tilde\omega_j-\partial_t\tilde\theta_j)
    \psld {\tilde R_j}{i\partial_x\tilde R_j}
    +(\tilde c_j-\partial_t\tilde
    x_j)\norm{\partial_x\tilde
      R_j}_{L^2}^2
    \\+\partial_t\tilde \omega_j\psld{i(\partial_\omega\tilde
      R_j+\mu_j\partial_c\tilde R_j)}{i\partial_x\tilde R_j}
    +(1+\abs{\Mod(t)})O\left(e^{-\frac{1}{4}\omega_\star L}\right)+
    O\left(\norm{\eps}_{H^1}\right).
  \end{multline*}
  For the first term in the left
  hand side of~\eqref{eq:third-orth}, there are some 
  cancellations, and we get
  \begin{multline*}
    -\psld{i\eps_t}{i \tilde R_j-\mu_j\partial_x\tilde
      R_j}=
    \partial_t\tilde \omega_j\psld{i(\partial_\omega\tilde
      R_j+\mu_j\partial_c\tilde R_j)}{i \tilde R_j-\mu_j\partial_x\tilde
      R_j}\\
    +(1+\abs{\Mod(t)})O\left(e^{-\frac{1}{4}\omega_\star L}\right)+ O\left(\norm{\eps}_{H^1}\right).
  \end{multline*}
  Remark that, in contrary to what happens in~\cite{MaMeTs06},
  the $\eps$ term is still of order $1$ and not of order $2$ (unless
  $\mu=0$, but this is ruled out by our assumptions on the $(c_j)$).
  The attentive reader will have noticed the appearance of the same
  elements as in the matrix $D\Phi$. 
  We have indeed
  \[
  M
  \begin{pmatrix}
    \tilde\omega_j-\partial_t\tilde\theta_j\\
    \tilde c_j-\partial_t\tilde
    x_j\\
    \partial_t\tilde \omega_j
  \end{pmatrix}
  =(1+\Mod(t))\left(O\left
      (\norm{\eps}_{L^2}\right)+O\left(e^{-\frac{1}{4}\omega_\star L}\right)\right)
  , 
  \]
  for $M=(m_{kl})$, where
  \begin{align*}
    m_{11}&=\norm{\tilde
            R_j}_{L^2}^2,
    &
      m_{12}&=\psld{i\partial_x\tilde R_j}{\tilde R_j},
    \\
    m_{21}&=\psld {\tilde R_j}{i\partial_x\tilde R_j},
    &
      m_{22}&=\norm{\partial_x\tilde
              R_j}_{L^2}^2,
    \\
    m_{31}&=\psld{i \tilde R_j-\mu_j\partial_x\tilde
            R_j}{\tilde R_j},
    &
      m_{32}&=\psld{i \tilde R_j-\mu_j\partial_x\tilde
              R_j}{i\partial_x\tilde R_j},
    \\
    m_{13}&=\psld{i(\partial_\omega\tilde
            R_j+\mu_j\partial_c\tilde R_j)}{\tilde R_j},
    &
      m_{23}&= \psld{i(\partial_\omega\tilde
              R_j+\mu_j\partial_c\tilde R_j)}{i\partial_x\tilde R_j},
    \\
    m_{33}&=\psld{i(\partial_\omega\tilde
            R_j+\mu_j\partial_c\tilde R_j)}{i \tilde R_j-\mu_j\partial_x\tilde
            R_j}.
  \end{align*}
  Using the explicit values of the coefficients $(m_{kl}$) (see above
  calculations for $D\Phi$), we obtain
  \[
  M=
  \begin{pmatrix}
    8 \arctan \left( \sqrt{\frac{2\sqrt{\omega_j}+c_j}{2\sqrt{\omega_j}-c_j}}
    \right) 
    &
    2\sqrt{4\omega_j-c_j^2}
    &
    *
    \\
    2\sqrt{4\omega_j-c_j^2}
    &
    8\omega_j \arctan \left( \sqrt{\frac{2\sqrt{\omega_j}+c_j}{2\sqrt{\omega_j}-c_j}}
    \right) 
    &
    *
    \\
    0
    &
    0
    &
    \frac{-\left( \frac{\tilde c_j}{\tilde \omega_j}-4\mu_j+\tilde c_j\mu_j^2
      \right)}{\sqrt{4\tilde \omega_j-\tilde c_j^2}}
  \end{pmatrix}.
  \]
  Hence $M$ is invertible using the same arguments as to prove that $D\Phi$ is, see in
  particular~\eqref{eq:help-det}, and we can infer that 
  \[
  \sum_{j=1}^N\abs*{  \tilde\omega_j-\partial_t\tilde\theta_j}+
  \abs*{\tilde c_j-\partial_t\tilde
    x_j}+
  \abs*{\partial_t\tilde \omega_j}
  \lesssim \norm{\eps}_{H^1}+e^{-\frac{1}{4}\omega_\star L}
  .
  \]
  This concludes the proof. 
\end{proof}

We complete this section by giving estimates on the interaction between $\tilde R_j$ and $\tilde
R_k$ when $j\neq k$.

\begin{lemma}[Interaction estimates
  One]\label{lem:interaction-estimates} 
  There exists a function $g\in L^\infty_tL^1_x(\R,\R^d)\cap
  L^\infty_tL^\infty_x(\R,\R^d)$ such that for all $j,k=1,\dots,N$ such
  that $j\neq k$, we have
  \begin{align*}
    |\tilde R_j|+|\partial_x\tilde R_j|+|\partial_\omega\tilde
    R_j|+|\partial_c\tilde R_j|&\lesssim e^{-\omega_\star|x-\tilde x_j|}.
    \\
    e^{-\omega_\star|x-\tilde x_j|}e^{-\omega_\star|x-\tilde x_k|}&\leq
                                                                    e^{-\frac12\omega_\star|\tilde x_j-\tilde x_k|}g(t,x).
  \end{align*}
  Moreover, if~\eqref{eq:control}-\eqref{eq:dynamical} hold, then
  \[
  \abs{\tilde x_j-\tilde x_k}\geq \frac 12(L+c_\star t). 
  \]
\end{lemma}

\begin{proof}
  We first remark that, by the exponential decay of the
  soliton profiles~\eqref{eq:exp-decay} and the definition
~\eqref{eq:def-omega} of $\omega_\star$, for $j=1,\dots,N$, we have
  \[
  |\tilde R_j|+|\partial_x\tilde R_j|+|\partial_\omega\tilde
  R_j|+|\partial_c\tilde R_j|\lesssim e^{-\omega_\star|x-\tilde x_j|}.
  \]
  There exists $g\in L^\infty_tL^1_x(\R,\R^d)\cap
  L^\infty_tL^\infty_x(\R,\R^d)$ such that
  \[
  e^{-\omega_\star|x-\tilde x_j|}e^{-\omega_\star|x-\tilde x_k|}\leq
  e^{-\frac12\omega_\star|\tilde x_j-\tilde x_k|}g(t,x),
  \]
  Indeed, let
  \[
  g_{jk}(t,x)=e^{-\frac12\omega_\star(|x-\tilde x_j|+|x-\tilde
    x_k|)}.
  \]
  One can then take
  \[
  g(t,x)=\sum_{j\neq k}g_{j,k}(t,x).
  \]
  In view of~\eqref{eq:control}-\eqref{eq:dynamical}, we can chose
  $\delta$ and $1/L$ small enough such that for any $j\neq k$ we have
  \[
  \abs{\tilde x_j-\tilde x_k}\geq \frac L2+\frac12c_\star t. 
  \]
  This concludes the proof. 
\end{proof}

\section{Monotonicity of Localized Conservations Laws}
\label{sec:local-proc}

We are using an energy technique to control the main difference $\eps$
between $u$ and the sum of modulated solitons $\sum\tilde R_j$. The energy technique
consists in using the coercivity of a linearized action functional related to the
conservation laws and the solitons. It can be viewed as a
generalization of the method used to prove stability of a single
soliton. The main difference when considering a sum of solitons is
that we need to introduce a localization procedure around each
soliton. We recover this way the desired coercivity property, but the
price to pay is that the quantities involved are no longer
conserved. Controlling their variations in time becomes a main issue,
which is dealt with using monotonicity properties. 

The localization
procedure is the following.
Let $\psi:\R\to[0,1]$ be a smooth cut-off function such that 
\begin{gather*}
  \psi(x)=0\; \text{for }x\leq -1,\quad\psi(x)=1\;\text{for
  }x\geq1,\quad\psi'(x)>0\;\text{ for }x\in(-1,1),
  \\
  (\psi'(x))^2\lesssim \psi(x),\quad (\psi''(x))^2\lesssim \psi'(x),\quad
  \text{for all }x\in\R.
\end{gather*}
For $j=2,\dots,N$, set
\begin{equation}
  \label{eq:def-sigma}
  \tilde\sigma_j=2\frac{\tilde\omega_j(0)-\tilde\omega_{j-1}(0)}{\tilde
    c_j(0)-\tilde c_{j-1}(0)}.
\end{equation}
Recall that $(\omega_j)$ and $(c_j)$ verify the speed-frequency ratio
assumption~\eqref{eq:speed-frequency-ratios}. Therefore, for $\delta$ and
$1/L$ small enough and by the estimate on the modulation parameters at
initial time~\eqref{eq:initial-time} we have 
\[
\max(c_{j-1},\tilde c_{j-1}(0))<\tilde\sigma_j<\min(c_j,\tilde c_j(0)),\quad j=2,\dots,N.
\]
Set 
\[
x_j^\sigma=\frac{\tilde x_{j-1}(0)+\tilde x_j(0)}{2},\quad a=\frac{L^2}{64},
\]
and define 
\[
\psi_1\equiv 1,\quad 
\psi_j(t,x)=\psi\left(\frac{x-x_j^\sigma-\tilde\sigma_jt}{\sqrt{t+a}}\right),\;j=2,\dots,N,\quad
\psi_{N+1}\equiv 0.
\]
We define the cut-off functions around the $j$-th solitary wave
by
\[
\chi_j(t,x)=\psi_j(t,x)-\psi_{j+1}(t,x), \quad j=1,\dots,N.
\]
The reason for the introduction of cut-off functions of this form will
become clear in the proof of the monotonicity properties.

We define the following functional, which is made by the combination
of localized masses and momenta around each solitary wave, weighted
with the corresponding modulated parameters 
$\tilde \omega_j(0)$ and $\tilde c_j(0)$
at
$t=0$:
\begin{equation}
  \label{eq:7}
  \mathcal I(t)=\frac12\sum_{j=1}^N\int_{\R}\left(\tilde
    \omega_j(0)|u|^2+\tilde c_j(0)\Im
    (u\bar u_x)\right)\chi_jdx.
\end{equation}

The following monotonicity property for $\mathcal I$ will be a key feature of the proof of Theorem~\ref{thm:1}.

\begin{proposition}
  [Monotonicity One]
  \label{prop:monotonicity-one}
  If $\delta$ and $1/L$ are small enough,
  then for all $t\in [0,t^\star]$, we have
  \[
  \mathcal I(t)- \mathcal I(0)\lesssim
  \frac{1}{L}\sup_{s\in[0,t]}\norm{\eps(s)}_{L^2}^2
  + e^{-\frac{1}{16}\omega_\star\left( c_\star^\sigma t+L\right)}.
  \]
\end{proposition}

To prove Proposition~\ref{prop:monotonicity-one}, it is convenient to
rewrite $\mathcal I$ using the 
functionals defined for $j=2,\dots,N$ by
\[
\mathcal I_j(t)=\frac12\int_{\R}\left(\frac{\tilde\sigma_j}{2}|u|^2+\Im(u \bar u_x) \right) \psi_j dx.
\]

\begin{lemma}
  [Decomposition of the functional $\mathcal{I}$]
  \label{lem:decomposition}
  We have
  \[
  \mathcal  I(t)=\tilde\omega_1(0)M(u)+\tilde
  c_1(0)P(u)+\sum_{j=2}^N(\tilde c_j(0)-\tilde c_{j-1}(0))\mathcal I_j(t).
  \]
\end{lemma}

The proof of Lemma~\ref{lem:decomposition} consists in a simple
rearrangement of the sum in the definition~\eqref{eq:7} of
$\mathcal{I}$ using the definition~\eqref{eq:def-sigma} of $\tilde\sigma_j$. We omit the details.

Proposition~\ref{prop:monotonicity-one} is  a consequence of Lemma
\ref{lem:decomposition}, the conservation of mass and momentum, and the
following monotonicity result for each of the functionals $\mathcal
I_j$.

\begin{proposition}[Monotonicity Two]\label{prop:monotonicity-two}
  If $\delta$ and $1/L$ are small enough,
  then for any $j=2,\dots,N$ and $t\in [0,t^\star]$ we
  have
  \[
  \mathcal I_j(t)-\mathcal I_j(0)\lesssim \frac{1}{L}\sup_{s\in
    [0,t]}\norm{\eps(s)}^2_{L^2}+  e^{-\frac{1}{16}\omega_\star\left( c_\star^\sigma t+L\right)}.
  \]
\end{proposition}

\begin{proof}[Proof of Proposition~\ref{prop:monotonicity-two}]
  Fix $j\in\{2,\dots,N\}$. To express the time derivative of $\mathcal I_{j
  }$ in a form to which we 
  can give a sign, we will use a Galilean transformation. 
  We define $v$ by
  \[
  u(t,x)=e^{i\frac{\tilde\sigma_j }{2}\left(x-x_j^\sigma-\frac{\tilde\sigma_j }{2}t\right)}v(t,x-x_j^\sigma-\tilde\sigma_j  t).
  \]
  We insist on the fact that since~\eqref{eq:nls} is not Galilean
  invariant, \emph{$v$ is not a solution
    of~\eqref{eq:nls} anymore}. It satisfies the modified equation
  \[
  iv_t+v_{xx}+i|v|^2v_x-\frac {\tilde\sigma_j}2|v|^2v=0,
  \]
  and we have
  \[
  \mathcal I_{{j}}(t)=\frac12\int_{\R}\Im(v \bar v_x) \psi \left(\frac{x}{\sqrt{t+a}}\right) dx.
  \]
  One realizes that the advantage of introducing $\mathcal I_{{j}}$ is that there is
  now no mass factor in the expression of $\mathcal I_j$ in terms of $v$.
  Computing the time derivative, we obtain
  \begin{multline*}
    \frac{\partial}{\partial t}
    \mathcal I_{{j}}(t)=
    -\frac{1}{\sqrt{t+a}}\int_{\R}\left(|v_x|^2+\frac12\Im\left(|v|^2\bar
        vv_x\right)+\frac{{\tilde\sigma_j}}{8}|v|^4 \right) \psi_x
    \left(\frac{x}{\sqrt{t+a}}\right) dx\\
    +\frac{1}{4(t+a)^{\frac32}}\int_{\R}|v|^2\psi_{xxx}\left(\frac{x}{\sqrt{t+a}}\right)
    dx\\
    -\frac{1}{4(t+a)^{\frac32}}\int_{\R}\Im(xv \bar v_x) \psi_x \left(\frac{x}{\sqrt{t+a}}\right) dx.
  \end{multline*}
  We distribute the last term between the quadratic terms. Using
 Young's inequality, we have
  \begin{multline*}
    \left|\frac{1}{4(t+a)^{\frac32}}\int_{\R}\Im(xv \bar v_x) \psi_x
      \left(\frac{x}{\sqrt{t+a}}\right) dx\right|\leq \\
     \frac{1}{8\sqrt{t+a}}\int_{\R}|v_x|^2 \psi_x
    \left(\frac{x}{\sqrt{t+a}}\right) dx+\frac{1}{8 (t+a)^{\frac32}}\int_{\R}|v|^2\left(\frac{x}{\sqrt{t+a}}\right)^2 \psi_x
    \left(\frac{x}{\sqrt{t+a}}\right) dx.
  \end{multline*}
  In addition, since $\psi_x$ is supported on $[-1,1]$, we have 
  \[
  \frac{1}{8 (t+a)^{\frac32}}\int_{\R}|v|^2\left(\frac{x}{\sqrt{t+a}}\right)^2 \psi_x
  \left(\frac{x}{\sqrt{t+a}}\right) dx\leq \frac{1}{8 (t+a)^{\frac32}}\int_{\R}|v|^2 \psi_x
  \left(\frac{x}{\sqrt{t+a}}\right) dx.
  \]
  We also apply Young's inequality to the derivative part of the nonlinear term:
  \begin{multline*}
    \left|
      \frac{1}{\sqrt{t+a}}\int_{\R}\frac12\Im\left(
        |v|^2\bar
        vv_x\right) \psi_x
      \left(
        \frac{x}{\sqrt{t+a}}
      \right) dx\right|\leq \\
    \frac{1}{8\sqrt{t+a}}\int_{\R}|v_x|^2 \psi_x
    \left(\frac{x}{\sqrt{t+a}}\right) dx+\frac{1}{2\sqrt{t+a}}\int_{\R}|v|^6 \psi_x \left(\frac{x}{\sqrt{t+a}}\right) dx.
  \end{multline*}
  Summarizing, we have obtained that
  \begin{multline}
    \label{eq:6}
    \frac{\partial}{\partial t}
    \mathcal I_{{j}}(u)\leq 
    -\frac{1}{\sqrt{t+a}}\int_{\R}\left(\frac34|v_x|^2-\frac12|v|^6+\frac{{\tilde\sigma_j}}{8}|v|^4 \right) \psi_x
    \left(\frac{x}{\sqrt{t+a}}\right) dx\\
    +\frac{1}{4(t+a)^{\frac32}}\int_{\R}|v|^2\left(\frac12\psi_x+\psi_{xxx}\right)\left(\frac{x}{\sqrt{t+a}}\right)
    dx.
  \end{multline}
  By assumption~\eqref{eq:speed-frequency-ratios} we have $\sigma_j>0$, thus we also have
  $\tilde\sigma_j>0$ (for
  $\delta$ and $1/L$ small enough). Therefore, to obtain the (quasi)-monotonicity of
  $\mathcal I_j$, it is sufficient to bound the $L^2$-term and the nonlinear term
  with power $6$. This is allowed by the following claims.

  \begin{claim}\label{claim:1}
    For $\delta$ and $1/L$ small enough and for any $t\in[0,t^\star]$, we
    have
    \[
    \int_{|x|<\sqrt{t+a}}|v|^2dx\leq  2\norm{\eps}_{L^2}^2+O\left(e^{-\frac{1}{16}\omega_\star\left( c_\star^\sigma t+L\right)}\right).
    \] 
  \end{claim}

  \begin{claim}\label{claim:2}
    For $\delta$ and $1/L$ small enough and for any $t\in[0,t^\star]$, we
    have
    \begin{multline*}
      \int_{\R}|v|^6 \psi_x\left(\frac{x}{\sqrt{t+a}}\right) dx\\
      \leq \frac14\int_{\R}|v_x|^2\psi_x
      \left(\frac{x}{\sqrt{t+a}}\right) dx+\frac{1}{t+a}\norm{\eps}_{L^2}^2+O\left(e^{-\frac{1}{16}\omega_\star\left( c_\star^\sigma t+L\right)}\right).
    \end{multline*}
  \end{claim}

  \begin{proof}[Proof of Claim~\ref{claim:1}]
    By definition of $v$ as a Galilean transform of $u$, we have 
    \begin{multline}
      \label{eq:4}
      |v(t,x)|^2=|u(t,x+x_j^\sigma+\tilde\sigma_j t)|^2
\\
\leq 2\sum_{k=1}^N \left|\tilde
        R_k(t,x+x_j^\sigma+\tilde\sigma_j t)\right|^2+2\left|\eps(t,x+x_j^\sigma+\tilde\sigma_j t)\right|^2.
    \end{multline}
    By exponential decay of the solitons profiles and the control
~\eqref{eq:control} on the
    modulation parameters $(\tilde \omega_k)$ and $(\tilde c_k)$, we have
    \[
    \sum_{k=1}^N \left|\tilde
      R_k(t,x+x_j^\sigma+\tilde\sigma_j t)\right|^2\lesssim \sum_{k=1}^N
    e^{-\omega_\star|x-\tilde x_k(t)+x_j^\sigma+\tilde\sigma_j t|}.
    \]
    Assume that $|x|<\sqrt{t-a}$.
    We have 
    \[
    |x-\tilde x_k(t)+x_j^\sigma+\tilde\sigma_j t|\geq |\tilde x_k(t)-x_j^\sigma-\tilde\sigma_j t|-|x|,
    \]
    which for $|x|<\sqrt{t+a}<\sqrt{t}+\sqrt{a}=\sqrt{t}+\frac L8$ gives
    \[
    |x-\tilde x_k(t)+x_j^\sigma+\tilde\sigma_j t|\geq |\tilde
    x_k(t)-x_j^\sigma-\tilde\sigma_j t|-\sqrt{t}-\frac L8.
    \]
    If $k\geq j$, then using  the dynamical system~\eqref{eq:dynamical} verified by the
    modulation parameters we get  
    \[
    \partial_t\tilde x_k\geq
    \partial_t \tilde  x_j\geq c_j-\frac{c_\star^\sigma}{8}.
    \]
    Since in addition $\tilde x_k(0)\geq \tilde x_j(0)$, we have 
    \[
    \tilde x_k(t)-x_j^\sigma-\tilde\sigma_j t\geq
    \left(c_j-\frac{c_\star^\sigma}{8}\right)t+\tilde x_j(0)
    -x_j^\sigma-\tilde\sigma_j t
    \geq \frac{c_\star^\sigma}{8}t+\frac{\tilde x_j(0)-\tilde x_{j-1}(0)}{2},
    \]
    where for the last inequality we have used $x_j^\sigma=(\tilde
    x_j(0)+\tilde x_{j-1}(0))/2$ and
    $c_j-\sigma_j\geq \frac{c_\star^\sigma}{4}$. Therefore, using
    $\tilde x_j(0)>\tilde x_{j-1}(0)+L/2$, we get
    \begin{equation}
      \label{eq:2}
      |x-\tilde x_k(t)+x_j^\sigma+\tilde\sigma_j t|\geq \frac{c_\star^\sigma}{8}t-\sqrt{t}+\frac L8.
    \end{equation}
    Choose now $L$ large enough so that
    $\min_{t\geq 0}\left(\frac{c_\star^\sigma}{16}t-\sqrt{t}+\frac
      {L}{16}\right)>0$. Then
    \[
    \frac{c_\star^\sigma}{8}t-\sqrt{t}+\frac
    L8=\frac{c_\star^\sigma}{16}t+\frac
    {L}{16}+\left(\frac{c_\star^\sigma}{16}t-\sqrt{t}+\frac
      {L}{16}\right)\geq \frac{c_\star^\sigma}{16}t+\frac
    {L}{16},
    \]
    and we infer from~\eqref{eq:2} that 
    \[
    |x-\tilde x_k(t)+x_j^\sigma+\tilde\sigma_j t|\geq \frac{c_\star^\sigma}{16}t+\frac
    {L}{16}.
    \]
    Arguing in a similar fashion for $k\leq j-1$ allows us to obtain, for
    $L$ large enough,
    \begin{equation}
      \label{eq:5}
      \sum_{k=1}^N \left|\tilde
        R_k(t,x+x_j^\sigma+\tilde\sigma_j t)\right|^2\lesssim \sum_{k=1}^N
      e^{-\frac{1}{16}\omega_\star\left( c_\star^\sigma t+L\right)}.
    \end{equation}
    Combining~\eqref{eq:4} and~\eqref{eq:5} gives the desired conclusion.
  \end{proof}

For further reference, we state here a Lemma which can be obtained using
similar arguments as in Claim~\ref{claim:1}.
\begin{lemma}[Interaction Estimates
  Two]\label{lem:interaction-est-2}
There exists a function $g\in L^\infty_tL^1_x(\R,\R^d)\cap
L^\infty_tL^\infty_x(\R,\R^d)$ such that for all $j,k=1,\dots,N$ such
that $j\neq k$, we have
\begin{align*}
e^{-\omega_\star|x-\tilde x_j|}\chi_k(t,x)&\leq
e^{-\frac1{32}\omega_\star(c_\star^\sigma t+L)}g(t,x),
\end{align*}
  where 
  \begin{equation}
    \label{eq:9}
    c^\sigma_\star=\min\left\{ |\tilde\sigma_j-c_k|;\;j,k=1,\dots,N,\,
      j\neq 1 \right\}.
  \end{equation}
\end{lemma} 

Let us recall without proof the following technical lemma from~\cite{MaMeTs06,Me01}
that is used for the proof of Claim~\ref{claim:2}.

\begin{lemma}\label{lem:technical}
  Let $w\in H^1(\R)$ and let $h\geq 0$ be a $\mathcal C^2$ bounded
  function such that $\sqrt{h}$ is $\mathcal C^1$ and $(h_x)^2\lesssim
  h$. Then
  \[
  \int_{\R}|w|^6hdx\leq 8\left(\int_{\mathrm{supp}(h)}|w|^2dx\right)^2\left(\int_{\R}|w_x|^2hdx+\int_{\R}|w|^2\frac{(h_x)^2}{h}dx\right),
  \]
  where $\mathrm{supp}(h)$ denotes the support of $h$.
\end{lemma}

  \begin{proof}[Proof of Claim~\ref{claim:2}]
    From the technical result Lemma~\ref{lem:technical}, we infer that 
    \begin{multline*}
      \int_{\R}|v|^6\psi_x\left(\frac{x}{\sqrt{t+a}}\right)dx
      \leq 8\left(\int_{|x|<\sqrt{t+a}}|v|^2dx\right)^2 \times
      \\\times\left(\int_{\R}|v_x|^2\psi_x\left(\frac{x}{\sqrt{t+a}}\right)dx+\frac{1}{t+a}\int_{\R}|v|^2\frac{(\psi_{xx})^2}{\psi_x}\left(\frac{x}{\sqrt{t+a}}\right)dx\right).
    \end{multline*}
    Using Claim~\ref{claim:1} and the fact that by construction
    $\frac{(\psi_{xx})^2}{\psi_x}\lesssim 1$, we get for $\delta$ and $1/L$
    small enough that
    \begin{multline*}
      \int_{\R}|v|^6\psi_x\left(\frac{x}{\sqrt{t+a}}\right)dx
      \\\leq
      \frac14\int_{\R}|v_x|^2\psi_x\left(\frac{x}{\sqrt{t+a}}\right)dx+\frac{1}{2(t+a)}\int_{|x|<\sqrt{t+a}}|v|^2dx,
      \\\leq \frac14\int_{\R}|v_x|^2\psi_x\left(\frac{x}{\sqrt{t+a}}\right)dx+\frac{1}{t+a}\norm{\eps}_{L^2}^2+O(e^{-\frac{1}{16}\omega_\star\left(
          c_\star^\sigma t+L\right)}).
    \end{multline*}
    This finishes the proof of Claim~\ref{claim:2}.
  \end{proof}

  Let us now conclude the
  proof of Proposition~\ref{prop:monotonicity-two}. Coming back to
~\eqref{eq:6} and using $\tilde\sigma_j>0$ and Claims~\ref{claim:1} and~\ref{claim:2}, we
  get
  \[
  \partial_t\mathcal I_j(t)\lesssim 
  \frac{1}{(t+a)^{\frac32}}\norm{\eps(t)}_{L^2}^2+ e^{-\frac{1}{16}\omega_\star\left( c_\star^\sigma t+L\right)}.
  \]
  Integrating between $0$ and $t$, we obtain (using in particular $a=L^2/64$)
  \[
  \mathcal I_j(t)-\mathcal I_j(0)\lesssim \frac{1}{L}\sup_{s\in[0,t^\star]}\norm{\eps(s)}_{L^2}^2+e^{-\frac{1}{16}\omega_\star\left( c_\star^\sigma t+L\right)},
  \]
  and this finishes the proof.
\end{proof}

For convenience, we also introduce here functionals similar to
$\mathcal I_j$ but with a different parameter $\sigma$. They will be
useful when we will control the modulation parameters. For
$j=2,\dots,N$, let $\tau_j$ be such that
\[
\tilde c_{j-1}(0)<\tau_j<\tilde c_j(0).
\]
For any $j=2,\dots,N$, we define
\begin{gather*}
  \psi_{j,\tau_j}(t,x)=\psi\left(\frac{x-x_j^\sigma-\tau_j
      t}{\sqrt{t+a}}\right),\quad
  \mathcal I_{j,\tau_j}(t)=\frac12\int_{\R}\left(\frac{\tau_j}{2}|u|^2+\Im(u \bar u_x) \right) \psi_{j,\tau_j} dx.
\end{gather*}
Then following the same proof as for Proposition
\ref{prop:monotonicity-two} we get the following result.
\begin{proposition}[Monotonicity Three]\label{prop:monotonicity-three}
  If $\delta$ and $1/L$ are small enough,
  then for any $j=2,\dots,N$ and $t\in [0,t^\star]$ we
  have
  \[
  \mathcal I_{j,\tau_j}(t)-\mathcal I_{j,\tau_j}(0)\lesssim \frac{1}{L}\sup_{s\in
    [0,t^\star]}\norm{\eps(s)}^2_{L^2}+ e^{-\frac{1}{16}\omega_\star\left( c_\star^{\tau} t+L\right)},
  \]
  where 
  \[
    c_\star^{\tau}=\min\left\{ |\tau_j-c_k|;\;j,k=1,\dots,N,\,j\neq 1 \right\}.
  \]
\end{proposition}

\section{Linearized Action Functional and Coercivity for $N$ Solitons}
\label{sec:line-acti-funct}

For $j=1,\dots,N$, we define an action functional related to the $j$-th
soliton by 
\[
S_j(v)=E(v)+\tilde\omega_j(0)M(v)+\tilde c_j(0)P(v).
\]

For the sum of $N$ solitons
we define an action-like functional $\mathcal{S}$ which will
correspond to $S_j$ locally around the $j$-th
soliton. The functional $\mathcal S$ is given by
\[
\mathcal{S}(t)=E(u(t))+\mathcal I(t),
\]
where $\mathcal I$ is the functional composed of localized masses and momenta
defined in~\eqref{eq:7}.

It is classical when working with solitons and related solutions
of nonlinear dispersive equations to introduce functionals related to the second variation of the
action. In our context, we will work with the functional $\mathcal H$,
obtained as follows. 

We set
\[
  \bar c_\star=\min\{ c_\star,c_\star^\sigma\},
\]
where $c_\star$ and $ c_\star^\sigma$ are
given by~\eqref{eq:c_star} and 
\eqref{eq:9}.

\begin{lemma}[Expansion of the action]
  \label{lem:expansion}
  For $t\in[0,t^\star]$, we have
  \begin{multline}\label{eq:exp0}
    \mathcal S(t)=\sum_{j=1}^NS_j(\phi_{\tilde\omega_j(0),\tilde c_j(0)})+\frac12\mathcal{H}(t)\\
    +\sum_{j=1}^NO\left(|\tilde\omega_j(t)-\tilde\omega_j(0)|^2\right)+o(\norm{\eps}_{H^1}^2)+O(e^{-\frac{1}{32}\omega_\star(\bar
      c_\star
      t+L)}),
  \end{multline}
  where
  \begin{multline*}
    \mathcal{H}(t)=\norm{\eps_x}_{L^2}^2+\sum_{j=1}^N\Im\int_{\R}
    \left(|\tilde
      R_j|^2\bar \eps\eps_x+\tilde R_j\partial_x\tilde
      R_j(\bar\eps)^2+\overline{\tilde R_j}\partial_x\tilde R_j|\eps|^2\right)dx
    \\+\sum_{j=1}^N\left(\tilde\omega_j(t)\int_{\R}|\eps|^2\chi_jdx+\tilde
      c_j(t)\Im \int_{\R}\eps\bar\eps_x\chi_jdx\right).
  \end{multline*}
\end{lemma}

\begin{proof}
  Writing $u=\sum_{j=1}^N\tilde R_j+\eps$, we expand in the components
  of $\mathcal S$. For the energy, we have 
\[
E(u)=E\bigg(\sum_{j=1}^N\tilde R_j\bigg)+E'\bigg(\sum_{j=1}^N\tilde R_j\bigg)\eps+\frac12\dual{E''\bigg(\sum_{j=1}^N\tilde R_j\bigg)\eps}{\eps}+o(\norm{\eps}_{H^1}^2).
\]
From Lemma~\ref{lem:interaction-estimates} (Interaction Estimates One), we have
\begin{align*}
E\bigg(\sum_{j=1}^N\tilde R_j\bigg)
&=\sum_{j=1}^N E\left(\tilde R_j\right)+O(e^{-\frac14\omega_\star
  (c_\star t+L)}),\\
E'\bigg(\sum_{j=1}^N\tilde R_j\bigg)\eps
&=\sum_{j=1}^N E'\left(\tilde R_j\right)\eps +O(e^{-\frac14\omega_\star
  (c_\star t+L)}),\\
\dual{E''\bigg(\sum_{j=1}^N\tilde R_j\bigg)\eps}{\eps}
&=\sum_{j=1}^N\dual{E''\left(\tilde R_j\right)\eps}{\eps}+O(e^{-\frac14\omega_\star
  (c_\star t+L)}).
\end{align*}
From Lemma~\ref{lem:interaction-est-2} (Interaction Estimates Two), we also have
\begin{align*}
  \frac12\int_{\R}|u|^2\chi_jdx&=\sum_{j=1}^NM(\tilde R_j)+2\sum_{j=1}^NM'(\tilde
                                 R_j)\eps+\frac12\int_{\R}|\eps|^2\chi_jdx +O(e^{-\frac{1}{32}\omega_\star
  (c_\star^\sigma t+L)}),\\
  \frac12\Im\int_{\R}u\bar u_x\chi_jdx&=\sum_{j=1}^NP(\tilde R_j)+2\sum_{j=1}^NP'(\tilde
                                 R_j)\eps+\frac12\Im\int_{\R}\eps\bar \eps_x\chi_jdx+O(e^{-\frac{1}{32}\omega_\star
  (c_\star^\sigma t+L)}).
\end{align*}
Recall that, for $j=1,\dots,N$, the function $\tilde R_j$ verifies
\[
E'(\tilde R_j)+\tilde\omega_j(t)M'(\tilde R_j)+\tilde c_j(t)P'(\tilde R_j)=0.
\]
Therefore, we have
\begin{multline}\label{eq:exp1}
  \mathcal S(t)=\sum_{j=1}^NS_j(\tilde
  R_j)+\sum_{j=1}^N(\tilde\omega_j(0)-\tilde \omega_j(t))M'(\tilde
  R_j)\eps +\sum_{j=1}^N(\tilde c_j(0)-\tilde c_j(t))P'(\tilde
  R_j)\eps 
\\
+\frac12\sum_{j=1}^N\dual{E''\left(\tilde R_j\right)\eps}{\eps}
  +\sum_{j=1}^N\tilde\omega_j(0) \frac12\int_{\R}|\eps|^2\chi_jdx
  +\sum_{j=1}^N\tilde c_j(0) \frac12\int_{\R}\eps\bar \eps_x\chi_jdx 
\\+
  O(e^{-\frac{1}{32}\omega_\star (bar c_\star t+L)})+o(\norm{\eps}_{H^1}^2).
\end{multline}
A Taylor expansion in $\tilde\omega_j$ gives (using that $\tilde R_j(0)$ is
a critical point of $S_j$)
\begin{equation}\label{eq:exp2}
S_j(\tilde R_j(t))=S_j(\tilde R_j(0))+O(|\tilde\omega_j(t)-\tilde\omega_j(0)|^2).
\end{equation}
Moreover, using phase and translation invariance, we have 
\begin{equation}\label{eq:exp3}
S_j(\tilde
R_j(0))=S_j(\phi_{\tilde \omega_j(0),\tilde c_j(0)}).
\end{equation}
Using the relation~\eqref{eq:wc} between $\tilde\omega$ and $\tilde c$
and the last orthogonality condition on $\eps$ in~\eqref{eq:ortho}, for all $j=1,\dots,N$ we
obtain
\begin{multline*}
  (\tilde\omega_j(0)-\tilde \omega_j(t))M'(\tilde R_j)\eps
  +(\tilde c_j(0)-\tilde c_j(t))P'(\tilde R_j)\eps\\=
  (\tilde\omega_j(0)-\tilde \omega_j(t))\left(M'(\tilde R_j)+\mu_j
    P'(\tilde R_j)\right)\eps
\\= (\tilde\omega_j(0)-\tilde \omega_j(t))\psld{\tilde R_j+i\mu_j\partial_x\tilde R_j}{\eps}=0.
\end{multline*}
From Young's inequality and using again  the relation
\eqref{eq:wc} between $\tilde\omega$ and $\tilde c$,  for all $j=1,\dots,N$  we have
\[
\abs*{(\tilde\omega_j(0)-\tilde \omega(t)) \frac12\int_{\R}|\eps|^2\chi_jdx
+(\tilde c_j(0)-c_j(t)) \frac12\int_{\R}\eps\bar \eps_x\chi_jdx }
\lesssim |\tilde\omega_j(0)-\tilde \omega(t)|^2+\norm{\eps}_{H^1}^4.
\]
Therefore, the term in the  second line of~\eqref{eq:exp1} becomes
\begin{multline*}
  \frac12\sum_{j=1}^N\dual{E''\left(\tilde R_j\right)\eps}{\eps}
  +\sum_{j=1}^N\tilde\omega_j(0) \frac12\int_{\R}|\eps|^2\chi_jdx
  +\sum_{j=1}^N\tilde c_j(0) \frac12\int_{\R}\eps\bar \eps_x\chi_jdx
  \\=\frac 12\mathcal H(t)+\sum_{j=1}^N O(|\tilde
  \omega_j(t)-\tilde\omega_j(0)|^2)+o(\norm{\eps}_{H^1}^2).
\end{multline*}
Together with~\eqref{eq:exp1},~\eqref{eq:exp2} and~\eqref{eq:exp3},
this gives the desired result~\eqref{eq:exp0}. 
\end{proof}

As a consequence of the coercivity of each linearized action $S_j''$
given by Proposition~\ref{prop:3}, we have global coercivity for
$\mathcal H$. 

\begin{lemma}
  [Coercivity]\label{lem:coercivity}
  There exists $\kappa>0$ such that 
  for all $t\in[0,t^\star]$ we have
  \[
  \mathcal{H}(t)\geq \kappa \norm{\eps}_{H^1}^2.
  \]
\end{lemma}

Before doing the proof, we explain how to control $\norm{\eps}_{H^1}$ using the coercivity property
Lemma~\ref{lem:coercivity} and the first monotonicity result
Proposition~\ref{prop:monotonicity-one}. 

\begin{lemma}\label{lem:energetic}
  For all $t\in[0,t^\star]$, we have
  \[
  \norm{\eps(t)}_{H^1}^2\lesssim
  \frac{1}{L}\sup_{s\in[0,t]}\norm{\eps(s)}_{L^2}^2+\norm{\eps(0)}
  _{H^1}^2+\sum_{j=1}^N|\tilde\omega_j(t)-\tilde\omega_j(0)|^2+e^{-\frac{1}{32}\omega_\star L}.
  \]
\end{lemma}

\begin{proof}
  From the Taylor-like expansion of $\mathcal S$ in Lemma
~\ref{lem:expansion} at $0$ and $t\in[0,t^\star]$, we get
  \begin{multline*}
    \mathcal S(t)-\mathcal S(0)
    =\frac12\left(\mathcal{H}(t)-\mathcal{H}(0)\right)
    \\
    +\sum_{j=1}^NO\left(|\tilde\omega_j(t)-\tilde\omega_j(0)|^2\right)+o(\norm{\eps(t)}_{H^1}^2)+o(\norm{\eps(0)}_{H^1}^2)+O(e^{-\frac{1}{32}\omega_\star
      L}).
  \end{multline*}
  By definition of $\mathcal S$ and conservation of the energy, we have
  \[
  \mathcal S(t)-\mathcal S(0)=\mathcal I(t)-\mathcal I(0).
  \]
  In addition, $\mathcal H$ is quadratic in $\eps$, hence
  \[
  \mathcal{H}(0)\lesssim \norm{\eps(0)}_{H^1}^2.
  \]
  Taking into account the coercivity of $\mathcal H$ given by Lemma
~\ref{lem:coercivity}, we obtain
  \begin{equation}
    \label{eq:3}
    \norm{\eps(t)}_{H^1}^2\lesssim
    \mathcal{H}(t)
    \lesssim
    \mathcal I(t)-\mathcal I(0)
    +\sum_{j=1}^N|\tilde\omega_j(t)-\tilde\omega_j(0)|^2+\norm{\eps(0)}_{H^1}^2+e^{-\frac{1}{32}\omega_\star
      L}.
  \end{equation}
  The conclusion then follows from the first monotonicity result
  Proposition~\ref{prop:monotonicity-one}.
\end{proof}

\begin{proof}[Proof of Lemma \ref{lem:coercivity}]
We write $\mathcal H(t)$ as
\begin{align*}
\mathcal H(t)=&\sum_{j=1}^N\int \big(|\varepsilon_x|^2+\tilde\omega_j(t)|\varepsilon|^2+\tilde c_j(t)\Im (\eps\bar\eps_x)\big)\chi_j(t)\,dx\\
 &\quad +\sum_{j=1}^N\Im\int_{\R}
    \left(|\tilde
      R_j|^2\bar \eps\eps_x+\tilde R_j\partial_x\tilde
      R_j(\bar\eps)^2+\overline{\tilde R_j}\partial_x\tilde R_j|\eps|^2\right)dx\\
    = & \sum_{j=1}^N\int \Big(|\varepsilon_x|^2+\tilde\omega_j(t)|\varepsilon|^2+\tilde c_j(t)\Im (\eps\bar\eps_x)\\
    &\quad+\Im
    \left(|\tilde
      R_j|^2\bar \eps\eps_x+\tilde R_j\partial_x\tilde
      R_j(\bar\eps)^2+\overline{\tilde R_j}\partial_x\tilde R_j|\eps|^2\right)\Big)\chi_j(t)\,dx+O\big(e^{-\frac{1}{16}\omega_\star L}\big)\|\varepsilon\|_{H^1}^2\\
      \triangleq & \sum_{j=1}^N \mathcal H_j(t) +O\big(e^{-\frac{1}{16}\omega_\star L}\big)\|\varepsilon\|_{H^1}^2
      .
\end{align*}
Now we need another cut-off function defined as in
\cite{MaMeTs06}. Let $\Lambda$ be a smooth positive function satisfying
\[
\Lambda=1 \mbox{ on } [-1,1], \quad \Lambda \sim e^{-|x|} \mbox{ on } \R,\quad |\Lambda'|\lesssim \Lambda.
\]
Moreover, we fix $B:1\ll B\ll L$, and denote $\Lambda_j(x)=\Lambda(\frac{x-\tilde x_j(t)}{B})$. Then for some small $c>0$,
\[
\chi_j\ge \Lambda_j-e^{-cL/B},\quad \mbox{supp} (\chi_j- \Lambda_j)\subset
\{x:|x-\tilde x_j(t)|\ge B\},\quad |\Lambda'_j|\lesssim \frac1B\Lambda_j.
\]
Let $z_j=\eps\sqrt{\Lambda_j}$. Then
\[
\left|\partial_x\eps\right|^2\Lambda_j=|\partial_xz_j|^2+\frac14|z_j|^2\left(\frac{\Lambda_j'}{\Lambda_j}\right)^2
-\Re\left(z_j\partial_x\bar z_j\frac{\Lambda_j'}{\Lambda_j}\right).
\]
This implies
\begin{multline}\label{11.45}
\int_\R|\partial_xz_j|^2dx-\frac
CB\int_R(|\partial_x z_j|^2+|z_j|^2)dx
\\\leq\int_\R\left|\partial_x\eps\right|^2\Lambda_jdx\leq
\\
\int_\R|\partial_xz_j|^2dx+\frac
CB\int_R(|\partial_x z_j|^2+|z_j|^2)dx.
\end{multline}
Moreover, we have
\[
\eps\partial_x\bar \eps=
\frac{z_j\partial_x\bar z_j}{\Lambda_j}-|z_j|^2\frac{\Lambda_j'}{2\Lambda_j^2},
\]
and therefore,
\begin{equation}\label{11.45-1}
\Im(\eps\partial_x\bar \eps)\Lambda_j=\Im(z\partial_x\bar z_j).
\end{equation}
Now, using the localization of $\tilde R_j$,  we rewrite $\mathcal{H}_j$ as
\begin{align}
\mathcal{H}_j(t)=&\int \Big[|\varepsilon_x|^2+\tilde\omega_j(t)|\varepsilon|^2+\tilde c_j(t)\Im (\eps\bar\eps_x)\notag\\
    &\quad+\Im
    \left(|\tilde
      R_j|^2\bar \eps\eps_x+\tilde R_j\partial_x\tilde
      R_j(\bar\eps)^2+\overline{\tilde R_j}\partial_x\tilde R_j|\eps|^2\right)\Big]\Lambda_j(t)\,dx\notag\\
&+\int \Big(|\eps_x|^2+\tilde \omega_j(t)|\eps|^2+\tilde c_j(t)\Im \big(\eps\bar\eps_x\big)\Big)(\chi_j-\Lambda_j)\,dx\notag\\
    &+O(e^{-\frac1{16}\omega_\star B})\|\eps\|_{H^1}^2.\label{16.32}
\end{align}
Using~\eqref{11.45} and~\eqref{11.45-1} we have
\begin{align}
&\int \Big[|\varepsilon_x|^2+\tilde\omega_j(t)|\varepsilon|^2+\tilde c_j(t)\Im (\eps\bar\eps_x)\notag\\
    &\quad+\Im
    \left(|\tilde
      R_j|^2\bar \eps\eps_x+\tilde R_j\partial_x\tilde
      R_j(\bar\eps)^2+\overline{\tilde R_j}\partial_x\tilde R_j|\eps|^2\right)\Big]\Lambda_j(t)\,dx\notag\\
\ge &\int \Big[|\partial_x z_j|^2+\tilde \omega_j(t)|z_j|^2+\tilde
      c_j(t)\Im \big(z_j \partial_x\bar z_j\big)\notag\\
    &+\Im
    \left(|\tilde
      R_j|^2\overline{ z_j}\partial_xz_j+\tilde R_j\partial_x\tilde
      R_j(\overline{ z_j})^2+\overline{\tilde R_j}\partial_x\tilde R_j|z_j|^2\right)\Big]\,dx -\frac CB\int |z_j|^2\,dx\notag\\
=&H_{\tilde \omega(t),\tilde c_j(t)}(z_j)-\frac CB\int |z_j|^2\,dx.\label{16.01}
\end{align}
From the the orthogonality conditions~\eqref{eq:ortho}, we have for $ j=1,\dots,N,$
  \begin{equation*}
    \abs*{\psld{z_j}{i\tilde R_{j}}}=\abs*{\psld{z_j}{\partial_x\tilde R_{j}}}
    =\abs*{\psld{z_j}{M'(\tilde R_j)+\mu_jP'(\tilde R_j)}}\lesssim e^{-\frac1{16}\omega_\star B}\|z_j\|_{L^2},
  \end{equation*}
  after suitable perturbation (for example, let $\tilde z_j=z_j+a_1i\tilde R_{j}+a_2\partial_x\tilde R_{j}+a_3(M'(\tilde R_j)+\mu_jP'(\tilde R_j))$ for
  $|a_j|\lesssim e^{-\frac1{16}\omega_\star B}\|z\|_{L^2}, j=1,2,3$), we obtain from Proposition~\ref{prop:3} and~\eqref{11.45}, that for some small constant $\kappa'>0$,
 \[
 H_{\tilde \omega(t),\tilde c_j(t)}(z_j)\ge \kappa'\int (|\partial_xz_j|^2+|z_j|^2)\,dx\ge \left(\kappa'-\frac CB\right)\int (|\partial_x\eps|^2+|\eps|^2)\Lambda_j\,dx.
 \]
Hence, from~\eqref{16.01}, we obtain that
\begin{multline}
\bigg|\int \Big[|\varepsilon_x|^2+\tilde\omega_j(t)|\varepsilon|^2+\tilde c_j(t)\Im (\eps\bar\eps_x)\\
  +\Im
    \left(|\tilde
      R_j|^2\bar \eps\eps_x+\tilde R_j\partial_x\tilde
      R_j(\bar\eps)^2+\overline{\tilde R_j}\partial_x\tilde R_j|\eps|^2\right)\Big]\Lambda_j(t)\,dx \bigg|\\
\ge  \kappa''\int (|\partial_x\eps|^2+|\eps|^2)\Lambda_j\,dx,\label{16.03}
\end{multline}
where $\kappa''=\kappa'-\frac {2C}B$.
Furthermore, since $\tilde c_j^2<4\tilde \omega_j$, we have
\[
\tilde \nu(|\partial_x\eps|^2+|\eps|^2)\ge|\partial_x\eps|^2+\tilde \omega_j(t)|\eps|^2+\tilde c_j(t)\Im \big(\bar{\eps }\partial_x\eps\big)\ge \nu (|\partial_x\eps|^2+|\eps|^2),
\]
for some small $\nu>0$ and $\tilde \nu=3(1+\omega_j^2)$.
Hence, using $\chi_j\ge \Lambda_j-e^{-cL/B}$, we find
\begin{multline}
 \int \Big[|\partial_x\eps|^2+\tilde \omega_j(t)|\eps|^2 +\tilde c_j(t)\Im \big(\bar{\eps }\partial_x\eps\big)\Big](\chi_j-\Lambda_j)\,dx\\
\ge
\nu' \int \big(|\partial_x\eps|^2+|\eps|^2\big)(\chi_j-\Lambda_j)\,dx,\label{16.31}
\end{multline}
where $\nu'=\nu-\widetilde \nu e^{-cL/B}$. Choosing $L$ large enough, we have $\nu'>0$.
Inserting~\eqref{16.03} and~\eqref{16.31} into~\eqref{16.32}, we obtain that
\[
\mathcal H_j(t)\ge 2\kappa\int \big(|\partial_x\eps|^2+|\eps|^2\big)\chi_j\,dx,
\]
where $\kappa=\frac12\min\{\nu', \kappa'' \}$.  Since
$\sum_{j=1}^N \chi_j=1$, this proves the lemma.
\end{proof}

\section{Control of the modulation parameters}
\label{sec:contr-modul-param}

With Lemma~\ref{lem:energetic} in hands, the only thing left to
control are the modulation parameters $\tilde \omega_j$. We prove the
following result.

\begin{lemma}\label{lem:control-modulation}
  For all $t\in[0,t^\star]$, we have
  \[
  \sum_{j=1}^N|\tilde\omega_j(t)-\tilde\omega_j(0)|\lesssim
  \sup_{s\in[0,t]}\norm{\eps(s)}_{H^1}^2+e^{-\frac{1}{32}\omega_\star
    L}.
  \]
\end{lemma}

Getting control over the modulation parameters is not an easy task for
Schr\"odinger like equations. Indeed, a useful tool for that aim are
monotonicity properties of localized conservation laws. For the
Korteweg-de Vries equation, the localized mass satisfies this
monotonicity property and can be used directly to control the
modulation parameters (see~\cite{MaMeTs02}). For~\eqref{eq:nls} and
(NLS),
the monotonicity is verified only for the momentum and 
nothing similar is available for the mass (or the energy). This is the reason why one has to use several cut-off
functions,  in order to transfer the information from the momentum to the $Q_j$-quantity defined below. This was one of the main
ideas introduced in~\cite{MaMeTs06}. The argument here is however more
involved due to our choice of orthogonality conditions.

We first make the following claim. 

\begin{claim}\label{claim:split}
  For all $t\in[0,t^\star]$, we have 
  \[
  \sum_{j=2}^N|\mathcal I_j(t)-\mathcal I_j(0)|\lesssim
  \frac{1}{L}\sup_{s\in[0,t]}\norm{\eps(s)}_{H^1}^2
  +\norm{\eps(0)}_{H^1}^2
  +\sum_{j=1}^N|\tilde\omega_j(t)-\tilde\omega_j(0)|^2+e^{-\frac{1}{32}\omega_\star L}.
  \]
\end{claim}

\begin{proof}
  From~\eqref{eq:3} in the proof of Lemma~\ref{lem:energetic}, the
  decomposition of $\mathcal I$ from Lemma~\ref{lem:decomposition},
  and conservation of mass and momentum, we get 
  \begin{multline}
    \label{eq:11}
    \norm{\eps(t)}_{H^1}^2\lesssim
    \sum_{j=2}^N(\tilde c_j(0)-\tilde c_{j-1}(0))\left(\mathcal I_j(t)-\mathcal I_j(0)\right)
    \\+\sum_{j=1}^N|\tilde\omega_j(t)-\tilde\omega_j(0)|^2+\norm{\eps(0)}_{H^1}^2+e^{-\frac{1}{32}\omega_\star
      L}.
  \end{multline}
  On one hand, for all $j=1,\dots,N$ such that  $\mathcal I_j(t)-\mathcal I_j(0)\geq 0$, by Proposition
~\ref{prop:monotonicity-two} we have 
  \begin{equation}\label{eq:111}
  |\mathcal I_j(t)-\mathcal I_j(0)|\lesssim \frac{1}{L}\sup_{s\in
    [0,t]}\norm{\eps(s)}^2_{L^2}+  e^{-\frac{1}{16}\omega_\star L}.
\end{equation}
  On the other hand, 
  since by assumption $c_j>c_{j-1}$, we have for $\delta,1/L$ small
  enough that $\tilde c_j(0)-\tilde c_{j-1}(0)>0$  for all
  $j=1,\dots,N$ . Thus  for all $j=1,\dots,N$ such that $ \mathcal
  I_j(t)-\mathcal I_j(0)<0$, 
~\eqref{eq:11}-\eqref{eq:111} imply
  \[
  |\mathcal I_j(t)-\mathcal I_j(0)|\lesssim \frac{1}{L}\sup_{s\in
    [0,t]}\norm{\eps(s)}^2_{H^1}+
  \sum_{j=1}^N|\tilde\omega_j(t)-\tilde\omega_j(0)|^2+\norm{\eps(0)}_{H^1}^2+e^{-\frac{1}{32}\omega_\star L}.
  \]
  Combining these two facts gives the desired conclusion.
\end{proof}

As announced, we introduce the conserved quantity $Q_j$ combining the
mass and momentum:
\[
Q_j(u)=M(u)+\mu_{j}P(u).
\]
Remark  that, due to the choice of $\mu_j$ after
\eqref{eq:recall}, and using the explicit calculations~\eqref{FM:Mass-dw} and~\eqref{FM:Mom-dw}, we have
\begin{equation}
\frac{d}{d\tilde \omega_j}Q_j(\phi_{\tilde\omega_j,\tilde c_j})\Big|_{\tilde \omega_j=\omega_j}=
\left(\frac{c_j}{\omega_j}-4\mu_j+c_j\mu_j^2\right)\frac{1}{\sqrt{4\omega_j-c_j^2}}<0.
\label{23.03}
\end{equation}
Moreover, due to the orthogonality condition~\eqref{eq:ortho}, we have
\[
Q_j(\tilde R_j(t)+\eps)=Q_j(\tilde R_j(t))+O(\norm{\eps}_{H^1}^2).
\]
Using $(Q_j)_{j=1,\dots,N}$, we will be able to control the parameter $(\omega_j)_{j=1,\dots,N}$. To do this, we first need the following claim.

\begin{claim}\label{claim:Qj}
For any $j=1,\dots,N$ we have
\[
\Big|Q_j\left(\tilde R_j(t)\right)-Q_j\left(\tilde R_j(0)\right)\Big|
 \lesssim \sup_{s\in[0,t]}\norm{\eps(s)}_{H^1}^2
    +\sum_{k=1}^N|\tilde\omega_k(t)-\tilde\omega_k(0)|^2+e^{-\frac{1}{32}\omega_\star L}.
  \]
\end{claim}
\begin{proof}
Let $j=1,\dots,N$. We set the notations,
\[
m(u)=\frac12|u|^2; \quad p(u)=\frac12\Im
    (u\bar u_x);\quad
q_j(u)=m(u)+\mu_{j}p(u).
\]
Then we can rewrite $\mathcal I_j$ as
\begin{multline}
\mathcal I_j(t)=\int_{\R}\left(\frac12\tilde\sigma_jm(u)+p(u)\right)\psi_j\,dx\\
=\frac12\tilde\sigma_j\left(2\tilde\sigma_j^{-1}-\mu_j\right)
\int_{\R}\left(\left(2\tilde\sigma_j^{-1}-\mu_j\right)^{-1}q_j(u)+p(u)\right)\psi_j\,dx.\label{eq:y1}
\end{multline}
We assume that $2\tilde\sigma_j^{-1}-\mu_j\neq 0$ (otherwise, the formulas~\eqref{09.21}--\eqref{09.24} can be obtained more simply).
We introduce the constants
\[
\lambda_j=\left(\frac12\tilde\sigma_j\left(2\tilde\sigma_j^{-1}-\mu_j\right)\right)^{-1},\qquad
a_j=\left(2\tilde\sigma_j^{-1}-\mu_j\right)^{-1}.
\]
Then~\eqref{eq:y1} becomes
\begin{align}
\lambda_j\mathcal I_j(t)
=&
\int_{\R}\left(a_jq_j(u)+p(u)\right)\psi_j\,dx.\label{21.10}
\end{align}
Similarly, we rewrite $\mathcal I_{j,\tau_j}$ as
\begin{align*}
\lambda_{j,\tau_j}\mathcal I_{j,\tau_j}(t)
=&
\int_{\R}\left(a_{j,\tau_j} q_j(u)+p(u)\right)\psi_{j,\tau_j}\,dx,
\end{align*}
where  we have set
\[
\lambda_{j,\tau_j}=\left(\frac12\tau_j\left(2\tau_j^{-1}-\mu_j\right)\right)^{-1},\qquad
a_{j,\tau_j}=\left(2\tau_j^{-1}-\mu_j\right)^{-1}.
\]
In addition to~\eqref{21.10}, we also need another formula of
$\mathcal I_j(t)$ based on $q_{j-1}$ when $j\geq 2$.
\begin{align*}
\mathcal I_j(t)=\frac12\tilde\sigma_j\left(2\tilde\sigma_j^{-1}-\mu_{j-1}\right)
\int_{\R}\left(\left(2\tilde\sigma_j^{-1}-\mu_{j-1}\right)^{-1}q_{j-1}(u)+p(u)\right)\psi_j\,dx.
\end{align*}
Here, we also assume that $2\tilde\sigma_j^{-1}-\mu_{j-1}\neq 0$.
Let
\[
\gamma_j=\left(\frac12\tilde\sigma_j\left(2\tilde\sigma_j^{-1}-\mu_{j-1}\right)\right)^{-1},\qquad
b_{j}=\left(2\tilde\sigma_j^{-1}-\mu_{j-1}\right)^{-1},
\]
then
\begin{align*}
\gamma_j\mathcal I_j(t)
=&
\int_{\R}\left(b_{j}q_{j-1}(u)+p(u)\right)\psi_j\,dx.
\end{align*}
Similarly, we also rewrite $\mathcal I_{j,\tau_j}$ as
\begin{align*}
\gamma_{j,\tau_j}\mathcal I_{j,\tau_j}(t)
=&
\int_{\R}\left(b_{j,\tau_j} q_{j-1}(u)+p(u)\right)\psi_{j,\tau_j}\,dx,
\end{align*}
where we have set
\[
\gamma_{j,\tau_j}=\left(\frac12\tau_j\left(2\tau_j^{-1}-\mu_{j-1}\right)\right)^{-1},\qquad
b_{j,\tau_j}=\left(2\tau_j^{-1}-\mu_{j-1}\right)^{-1}.
\]
Take a small constant $\epsilon_0>0$,  and set the constants $\tau_j^1,\dots, \tau_j^4$ such that
\begin{align*}
a_{j,\tau_j^1}&=a_j+\epsilon_0,&a_{j,\tau_j^2}&=a_j-\epsilon_0,\\
b_{j+1,\tau_j^3}&=b_{j+1}+\epsilon_0,&b_{j+1,\tau_j^4}&=b_{j+1}-\epsilon_0.
\end{align*}
Then  for $j=1,\dots,N$ we have the following identity
\[
\lambda_{j,\tau_j^1}\mathcal I_{j,\tau_j^1}(t)-\lambda_j\mathcal
I_j(t)=\epsilon_0\int_{\R}
q_j(u)\psi_j\,dx+\int_{\R}(a_{j,\tau_j^1}q_j(u)+p(u))(\psi_{j,\tau_j}-\psi_j)dx.
\]
The function $(\psi_j-\psi_{j,\tau_j})$ is zero for $x<\tilde
  x_{j-1}+\frac{(c_{j-1}+\min(\tau_j,\tilde\sigma_j))}{2}t$ and for $x>\tilde
  x_{j}-\frac{(\max(\tau_j,\tilde\sigma_j)+ c_{j})}{2}t$.
  Hence,  due to the
  exponential localization of the solitons around each $\tilde x_k$
  and using the lower bound $L/2$ on the distance between $\tilde
  x_j-\tilde x_k$ given by~\eqref{eq:dist-mod-x}
  we have  
  \[
\abs*{\int_{\R}\left(a_{j,\tau_j^1}q_j\left(\sum_{k=1}^N\tilde R_k\right)+p\left(\sum_{k=1}^N\tilde R_k\right)\right)(\psi_{j,\tau_j}-\psi_j)dx}
  \lesssim e^{-\frac{1}{32}\omega_\star L}.
  \]
Similar estimates can be obtained replacing $\tau_j^1$ by $\tau_j^l$,
$l=2,3,4$ and $j$ by $j+1$. As a consequence, for $j=1,\dots,N$ we have the following identities,
\begin{align}
\lambda_{j,\tau_j^1}\mathcal I_{j,\tau_j^1}(t)-\lambda_j\mathcal I_j(t)&=\phantom{-}\epsilon_0\int_{\R} q_j(u)\psi_j\,dx+O(e^{-\frac{1}{32}\omega_\star L}),\label{09.21}\\
\lambda_{j,\tau_j^2}\mathcal I_{j,\tau_j^2}(t)-\lambda_j\mathcal
  I_j(t)&=-\epsilon_0\int_{\R}
          q_j(u)\psi_j\,dx+O(e^{-\frac{1}{32}\omega_\star
          L}),\label{09.22}
\end{align}
and for $j=1,\dots,N-1$ we have
\begin{align}
\gamma_{j+1,\tau_j^3}\mathcal I_{j+1,\tau_j^3}(t)-\gamma_{j+1}\mathcal I_{j+1}(t)&=\phantom{-}\epsilon_0\int_{\R} q_j(u)\psi_{j+1}\,dx+O(e^{-\frac{1}{32}\omega_\star L}),\label{09.23}\\
\gamma_{j+1,\tau_j^4}\mathcal I_{j+1,\tau_j^4}(t)-\gamma_{j+1}\mathcal I_{j+1}(t)&=-\epsilon_0\int_{\R} q_j(u)\psi_{j+1}\,dx+O(e^{-\frac{1}{32}\omega_\star L}).\label{09.24}
\end{align}
We assume that $\lambda_j,\gamma_{j+1}$ are both positive, the other
cases being treated similarly.
Choosing $\epsilon_0$ small enough, we can assume
$\lambda_{j,\tau_j^1}, \lambda_{j,\tau_j^2},
\gamma_{j+1,\tau_j^3},\gamma_{j+1,\tau_j^4}$ are also positive. Then
from~\eqref{09.21}--\eqref{09.24} for $j=1,\dots, N-1$, we have
\begin{align*}
&\epsilon_0\int_\R  q_j(u)\chi_j(t,x)\,dx\\
 &\quad = \left(\lambda_{j,\tau_j^1}\mathcal I_{j,\tau_j^1}(t)-\lambda_j\mathcal I_j(t)\right)+\left(\gamma_{j+1,\tau_j^4}\mathcal I_{j+1,\tau_j^4}(t)-\gamma_{j+1}\mathcal I_{j+1}(t)\right)+O(e^{-\frac{1}{32}\omega_\star L}),\\
&-\epsilon_0\int_\R  q_j(u)\chi_j(t,x)\,dx\\
 &\quad= \left(\lambda_{j,\tau_j^2}\mathcal I_{j,\tau_j^2}(t)-\lambda_j\mathcal I_j(t)\right)+\left(\gamma_{j+1,\tau_j^3}\mathcal I_{j+1,\tau_j^3}(t)-\gamma_{j+1}\mathcal I_{j+1}(t)\right)+O(e^{-\frac{1}{32}\omega_\star L}).
\end{align*}
Moreover, since $u=\sum_{k=1}^N \tilde R_k(t)+\eps$, and the support of $\chi_j$ is far away from the center of the soliton $\tilde R_k(t)$ when $k\neq j$, we have
\[
\int_\R  q_j(u)\chi_j\,dx=Q_j(\tilde R_j(t))+O(\|\eps\|_{H^1}^2)+O(e^{-\frac{1}{32}\omega_\star L}),
\]
where we have used the orthogonality conditions~\eqref{eq:ortho} to
cancel the first order term. 
Therefore,
\begin{multline*}
\epsilon_0\left(Q_j\left(\tilde R_j(t)\right)-Q_j\left(\tilde R_j(0)\right)\right)
=\epsilon_0\int_\R  q_j(u(t))\chi_j(t)\,dx-\epsilon_0\int_\R
q_j(u(0))\chi_j(0)\,dx\\
\shoveright{+O(\norm{\eps(t)}_{H^1}^2+\norm{\eps(0)}_{H^1}^2+e^{-\frac{1}{32}\omega_\star L})}\\
  = \left(\lambda_{j,\tau_j^1}\mathcal I_{j,\tau_j^1}(t)-\lambda_{j,\tau_j^1}\mathcal I_{j,\tau_j^1}(0)\right)+\left(\gamma_{j+1,\tau_j^4}\mathcal I_{j+1,\tau_j^4}(t)-\gamma_{j+1,\tau_j^4}\mathcal I_{j+1,\tau_j^4}(0)\right)\\
 -\left(\lambda_j\mathcal I_j(t)-\lambda_j\mathcal I_j(0)\right) -\left(\gamma_{j+1}\mathcal I_{j+1}(t)-\gamma_{j+1}\mathcal I_{j+1}(0)\right)\\
+O\Big(\sup_{s\in[0,t]}\norm{\eps(s)}_{H^1}^2+e^{-\frac{1}{32}\omega_\star L}\Big).
\end{multline*}
Now we use Proposition~\ref{prop:monotonicity-three} to control the
first and second terms,  and we use Claim~\ref{claim:split} to control
the third and fourth terms.  For $j=2,\dots, N-1$, we obtain 
\begin{align}
Q_j\left(\tilde R_j(t)\right)-Q_j\left(\tilde R_j(0)\right)\lesssim
\sup_{s\in[0,t]}\norm{\eps(s)}_{H^1}^2
    +\sum_{j=1}^N|\tilde\omega_j(t)-\tilde\omega_j(0)|^2+e^{-\frac{1}{32}\omega_\star L}.\label{17.33}
\end{align}
Similarly,
\begin{multline*}
-\epsilon_0\left(Q_j\left(\tilde R_j(t)\right)-Q_j\left(\tilde R_j(0)\right)\right)
=-\epsilon_0\int_\R  q_j(u(t))\chi_j(t)\,dx+\epsilon_0\int_\R
  q_j(u(0))\chi_j(0)\,dx\\
\shoveright{+O(\norm{\eps(t)}_{H^1}^2+\norm{\eps(0)}_{H^1}^2+e^{-\frac{1}{32}\omega_\star L})}\\
  = \left(\lambda_{j,\tau_j^2}\mathcal I_{j,\tau_j^2}(t)-\lambda_{j,\tau_j^2}\mathcal I_{j,\tau_j^2}(0)\right)+\left(\gamma_{j+1,\tau_j^3}\mathcal I_{j+1,\tau_j^3}(t)-\gamma_{j+1,\tau_j^3}\mathcal I_{j+1,\tau_j^3}(0)\right)\\
 -\left(\lambda_j\mathcal I_j(t)-\lambda_j\mathcal I_j(0)\right)
 -\left(\gamma_{j+1}\mathcal I_{j+1}(t)-\gamma_{j+1}\mathcal
   I_{j+1}(0)\right)\\
+O(\sup_{s\in[0,t]}\norm{\eps(s)}_{H^1}^2+e^{-\frac{1}{32}\omega_\star L}).
\end{multline*}
Hence, arguing as before, for $j=2,\dots,N-1$, we obtain
\begin{align}
-\left(Q_j\left(\tilde R_j(t)\right)-Q_j\left(\tilde R_j(0)\right)\right)\lesssim
\sup_{s\in[0,t]}\norm{\eps(s)}_{H^1}^2
    +\sum_{j=1}^N|\tilde\omega_j(t)-\tilde\omega_j(0)|^2+e^{-\frac{1}{32}\omega_\star L}.\label{17.34}
\end{align}
Now combining~\eqref{17.33} and~\eqref{17.34}, we get  for $j=2,\dots,N-1$,
\begin{align}
\Big|Q_j\left(\tilde R_j(t)\right)-Q_j\left(\tilde R_j(0)\right)\Big|\lesssim
\sup_{s\in[0,t]}\norm{\eps(s)}_{H^1}^2
    +\sum_{j=1}^N|\tilde\omega_j(t)-\tilde\omega_j(0)|^2+e^{-\frac{1}{32}\omega_\star L}.\label{22.40}
\end{align}
Now we remain to treat the case $j=1$ and $j=N$. For $j=N$, since
\[
\int_\R q_N(u)\psi_N\,dx=Q_N(\tilde R_N(t))+O\Big(\sup_{s\in[0,t]}\norm{\eps(s)}_{H^1}^2+e^{-\frac{1}{32}\omega_\star L}\Big),
\]
from~\eqref{09.21},  we have
\begin{align*}
&\epsilon_0\left(Q_N\left(\tilde R_N(t)\right)-Q_N\left(\tilde R_N(0)\right)\right)\\
=&\epsilon_0\int_\R  q_N(u(t))\psi_N(t)\,dx-\epsilon_0\int_\R  q_N(u(0))\psi_N(0)\,dx+O\Big(\sup_{s\in[0,t]}\norm{\eps(s)}_{H^1}^2+e^{-\frac{1}{32}\omega_\star L}\Big)\\
  = &\left(\lambda_{N,\tau_N^1}\mathcal I_{N,\tau_N^1}(t)-\lambda_{N,\tau_N^1}\mathcal I_{N,\tau_N^1}(0)\right) -\left(\lambda_N\mathcal I_N(t)-\lambda_N\mathcal I_N(0)\right)\\
& +O\Big(\sup_{s\in[0,t]}\norm{\eps(s)}_{H^1}^2+e^{-\frac{1}{32}\omega_\star L}\Big).
\end{align*}
Again, we use Proposition~\ref{prop:monotonicity-three} to control the first term,  and use Claim~\ref{claim:split} to control the second term, we obtain
\begin{align*}
Q_N\left(\tilde R_N(t)\right)-Q_N\left(\tilde R_N(t)\right)\lesssim
\sup_{s\in[0,t]}\norm{\eps(s)}_{H^1}^2
    +\sum_{j=1}^N|\tilde\omega_j(t)-\tilde\omega_j(t)|^2+e^{-\frac{1}{32}\omega_\star L}.
\end{align*}
Using~\eqref{09.22} instead,  we obtain from similar arguments that
\begin{align*}
-Q_N\left(\tilde R_N(0)\right)-Q_N\left(\tilde R_N(t)\right)\lesssim
\sup_{s\in[0,t]}\norm{\eps(s)}_{H^1}^2
    +\sum_{j=1}^N|\tilde\omega_j(t)-\tilde\omega_j(0)|^2+e^{-\frac{1}{32}\omega_\star L}.
\end{align*}\
Therefore, we get~\eqref{22.40} when $j=N$.

At last, we consider the case $j=1$.  We use~\eqref{09.23} and~\eqref{09.24} to get
\begin{align*}
-\epsilon_0\int_\R q_1(u)\psi_1\,dx=&-\epsilon_0Q_1(u)+\left(\gamma_{2,\tau_1^3}\mathcal I_{2,\tau_1^3}(t)-\gamma_2\mathcal I_2(t)\right),\\
\epsilon_0\int_\R q_1(u)\psi_1\,dx=&\epsilon_0Q_1(u)+\left(\gamma_{2,\tau_1^4}\mathcal I_{2,\tau_1^4}(t)-\gamma_2\mathcal I_2(t)\right).
\end{align*}
Then by mass and momentum conservation laws, and similar arguments as before, we also obtain~\eqref{22.40} when $j=1$.
\end{proof}

With these preliminaries out of the way, let us prove Lemma
\ref{lem:control-modulation}.

\begin{proof}[Proof of Lemma~\ref{lem:control-modulation}]

  From~\eqref{23.03} and~\eqref{eq:initial-time}, we
  know that for all $\tilde \omega$ and $\tilde c$ with the
  relationship~\eqref{eq:wc}, we have
\[
\frac{d}{d\tilde \omega_j}Q_j(\phi_{\tilde\omega_j,\tilde c_j})\neq 0.
\]
  Consequently, for any $j=1,\dots,N$ we have
  \[
    |\tilde \omega_j(t)-\tilde\omega_j(0)|\lesssim \Big|Q_j\left(\tilde R_j(t)\right)-Q_j\left(\tilde R_j(0)\right)\Big|
  \]
  Hence, from Claim~\ref{claim:Qj},  for $j=1,\dots,N$ , we obtain
  \[
    |\tilde \omega_j(t)-\tilde\omega_j(0)|\lesssim \sup_{s\in[0,t]}\norm{\eps(s)}_{H^1}^2
    +\sum_{k=1}^N|\tilde\omega_k(t)-\tilde\omega_k(0)|^2+e^{-\frac{1}{32}\omega_\star L}.
  \]
  This allows us to infer that for $j=1,\dots, N$ we have
  \[
  |\tilde \omega_j(t)-\tilde\omega_j(0)|\lesssim \sup_{s\in[0,t]}\norm{\eps(s)}_{H^1}^2
  +e^{-\frac{1}{32}\omega_\star L},
  \]
  and this concludes the proof of Lemma~\ref{lem:control-modulation}.
\end{proof}

\section{Proof of the Bootstrap Result}
\label{sec:proof-bootstr-result}

\begin{proof}[Proof of Proposition~\ref{prop:bootstrap}]
  With Lemmas~\ref{lem:energetic} and 
~\ref{lem:control-modulation} in hands, we can now conclude the proof
  of Proposition~\ref{prop:bootstrap}. For $\delta$ and $1/L$ small
  enough, we have 
  \[
  \norm{\eps(t)}_{H^1}^2\leq
  \frac{1}{2}\sup_{s\in[0,t]}\norm{\eps(s)}_{H^1}^2+C\norm{\eps(0)}
  _{H^1}^2+Ce^{-\frac{1}{32}\omega_\star L}.
  \]
  Therefore, for all $t\in[0,t^\star]$, we have 

  \begin{equation}
    \frac{1}{2}\sup_{s\in[0,t]}\norm{\eps(s)}_{H^1}^2\leq
    C\norm{\eps(0)}
    _{H^1}^2+Ce^{-\frac{1}{32}\omega_\star L}.
    \label{eq:18}
  \end{equation}
  Plugging that back in the control on the modulation parameters Lemma
~\ref{lem:control-modulation} we obtain
  \[
    \sum_{j=1}^N|\tilde\omega_j(t)-\tilde\omega_j(0)|
    \leq
    C\norm{\eps(0)}
    _{H^1}^2+Ce^{-\frac{1}{32}\omega_\star L}.
  \]
  From the modulation result Proposition~\ref{prop:modulation} and the
  bootstrap assumption we also
  have  at time $t=0$  the following estimate,
  \[
    \sum_{j=1}^N|\tilde\omega_j(0)-\omega_j|\leq \tilde C\delta.
 \]
  Recall also that
  \begin{equation}
\label{eq:21}
    \norm{\eps(0)}
    _{H^1}\leq \delta.
  \end{equation}
  We combine now~\eqref{eq:18}-\eqref{eq:21} to conclude the proof:
  \begin{multline*}
    \norm*{u-\sum_{j=1}^Ne^{\tilde\theta_j}\phi_j(\cdot-\tilde x_j)}_{H^1}
    \\
    \leq   \norm*{u-\sum_{j=1}^N\tilde
      R_j}_{H^1}+\norm*{\sum_{j=1}^N\tilde
      R_j-\sum_{j=1}^Ne^{\tilde\theta_j}\phi_j(\cdot-\tilde x_j)}_{H^1}
    \\
    \leq
    \norm{\eps}_{H^1}+C \sum_{j=1}^N|\tilde\omega_j(t)-\omega_j|
    \\
    \leq
    \norm{\eps}_{H^1}+C \sum_{j=1}^N|\tilde\omega_j(t)-\tilde\omega_j(0)|
    + C\sum_{j=1}^N|\tilde\omega_j(0)-\omega_j|
    \leq C_0\left(\delta+e^{-\frac{1}{32}\omega_\star L}\right).
  \end{multline*}
  Note that this last constant $C_0$ is \emph{independent} of
  $A_0$. Hence we may chose $A_0=2C_0$ and this concludes the proof. 
\end{proof}

\appendix

\section{Some Explicit Formulas}
\label{sec:some-techn-results}

In this section, we give explicit formulas for quantities evaluated on $\phi_{\omega,c}$.
Since they are obtained by elementary calculations using the explicit
formula~\eqref{eq:16}-\eqref{eq:varphi-formula} for $\phi_{\omega,c}$, we omit the details. 
We start with remarkable $L^p$-norms.
\begin{align}
\|\phi_{\omega,c}\|_{L^2}^2=&8\arctan\sqrt{\frac{2\sqrt{\omega}+c}{2\sqrt{\omega}-c}},\label{FM:L2}\\
\|\phi_{\omega,c}\|_{L^4}^4=&16c\arctan\sqrt{\frac{2\sqrt{\omega}+c}{2\sqrt{\omega}-c}}+8\sqrt{4\omega-c^2},\label{FM:L4}\\
\|\phi_{\omega,c}\|_{L^6}^6=&32(c^2+2\omega)\arctan\sqrt{\frac{2\sqrt{\omega}+c}{2\sqrt{\omega}-c}}+24c\sqrt{4\omega-c^2}. \label{FM:L6}
\end{align}
We also have
\begin{equation}\label{eq:norm-derivative-phi}
\norm{\partial_x\phi_{\omega,c}}_{L^2}^2=8\omega \arctan \sqrt{\frac{2\sqrt{\omega}+c}{2\sqrt{\omega}-c}}.
\end{equation}
The mass, momentum and  energy are given by
\begin{align}
M(\phi_{\omega,c})=&4\arctan\sqrt{\frac{2\sqrt{\omega}+c}{2\sqrt{\omega}-c}},\label{FM:Mass}\\
P(\phi_{\omega,c})=&\sqrt{4\omega-c^2},\label{FM:Mom}\\
E(\phi_{\omega,c})=&-\frac{c}{2}\sqrt{4\omega-c^2}. \label{FM:En}
\end{align}
Moreover, we have
\begin{align}
\partial_\omega M(\phi_{\omega,c})&=-\frac{c}{\omega\sqrt{4\omega-c^2}},\label{FM:Mass-dw}\\
\partial_c M(\phi_{\omega,c})=\partial_\omega P(\phi_{\omega,c})&=\frac{2}{\sqrt{4\omega-c^2}},\label{FM:Mom-dw}\\
\partial_c P(\phi_{\omega,c})&=-\frac{c}{\sqrt{4\omega-c^2}}. \label{FM:Mom-c}
\end{align}

\bibliographystyle{abbrv}
\bibliography{lecoz-wu}

\end{document}